\documentclass{amsart}

\usepackage[T1]{fontenc}
\usepackage[utf8]{inputenc}
\usepackage[USenglish]{babel}
\usepackage{textcase}

\usepackage[
	bookmarks=true,
	plainpages=false,
	linktocpage,
	colorlinks=true,
	citecolor=green!80!black,
	linkcolor=red!70!black,
	filecolor=magenta,
	urlcolor=magenta,
	breaklinks,
	pdfauthor={Martin Winter},
]{hyperref}

\usepackage{amsmath,amsthm}
\usepackage{calc, mathtools} 

\iftrue
\usepackage{amssymb} 
\else
\usepackage[charter,cal=cmcal]{mathdesign}
\fi

\usepackage{kbordermatrix}

\usepackage{colortbl,color} 
\usepackage[dvipsnames]{xcolor}
\usepackage{centernot} 
\usepackage{array} 
\usepackage{enumitem,moreenum} 
\usepackage[noadjust]{cite} 
\usepackage[nameinlink,capitalize,noabbrev]{cleveref}
\usepackage{nicefrac}

\usepackage{tikz-cd} 

\usepackage[font=small,labelfont=bf]{caption}

\usepackage{blkarray}

\usepackage{inconsolata}

\usepackage[
    colorinlistoftodos,
    backgroundcolor=yellow,
    textsize=footnotesize, 
]{todonotes}

\usepackage{wrapfig} 

\usepackage[normalem]{ulem} 
\newcommand{\mathsout}[1]{\ensuremath{\text{\sout{$#1$}}}}

\newcommand{\RR}{\mathbb{R}}    
                   
\newcommand{\Tsymb}{\top}
\newcommand{\T}{^{\Tsymb}}

\def\^#1{^{(#1)}}
\def\s^#1{^{\smash{(#1)}}}

\def\:{\colon}

\def\nlspace{\nolinebreak\space}
\def\nls{\nlspace}

\newcommand{\cupdot}{\mathbin{\mathaccent\cdot\cup}}
\newcommand{\precdot}{\mathrel{\prec\kern-0.9ex\cdot}}
\newcommand{\ldot}{\mathrel{<\kern-0.5ex\mathllap{\cdot}}}
\newcommand{\sldot}{\mathrel{<\kern-0.5ex\cdot\kern0.2ex}}
\newcommand{\gdot}{\mathrel{\cdot\kern-1.3ex>}}

\newcommand{\labelstyle}[1]{\upshape(\textit{#1})}
\newcommand{\mylabel}{\labelstyle{\roman*}}
\newenvironment{myenumerate}{\begin{enumerate}[label=\mylabel]}{\end{enumerate}}

\def\itm#1{{\labelstyle{\romannumeral#1\relax}}}

\newcommand{\freespace}{\kern.07em}

\newcommand{\ul}[1]{\underline{\smash{#1}}}

\newcommand{\bs}[1]{\boldsymbol{#1}}
\newcommand{\bsdot}[1]{\dot{\boldsymbol{#1}}}
\newcommand{\dotbs}[1]{\dot{\boldsymbol{#1}}}

\definecolor{NiceBlue}{rgb}{0.15, 0.2, 0.75}
\newcommand{\Def}[1]{\textcolor{NiceBlue!90}{\textit{#1}}}

\newcommand{\TODO}{{\footnotesize\textcolor{red}{TODO}}}

\makeatletter
\newcommand{\subalign}[1]{%
  \vcenter{%
    \Let@ \restore@math@cr \default@tag
    \baselineskip\fontdimen10 \scriptfont\tw@
    \advance\baselineskip\fontdimen12 \scriptfont\tw@
    \lineskip\thr@@\fontdimen8 \scriptfont\thr@@
    \lineskiplimit\lineskip
    \ialign{\hfil$\m@th\scriptstyle##$&$\m@th\scriptstyle{}##$\hfil\crcr
      #1\crcr
    }%
  }%
}
\makeatother

\newtheoremstyle{mythmstyle} 
    {\parsep}                    
    {\parsep}                    
    {\itshape}                   
    {}                           
    {\bfseries\scshape}          
    {.}                          
    {.5em}                       
    {}  
\newtheoremstyle{mydefstyle} 
    {\parsep}                    
    {\parsep}                    
    {}                   
    {}                           
    {\mdseries\scshape}          
    {.}                          
    {.5em}                       
    {}  

\numberwithin{equation}{section}

\theoremstyle{theorem}
\newtheorem{theorem}{Theorem}[section]
\newtheorem{corollary}[theorem]{Corollary}
\newtheorem{lemma}[theorem]{Lemma}

\newtheorem{conjecture}[theorem]{Conjecture}

\theoremstyle{definition}

\newtheorem{remark}[theorem]{Remark}

\crefname{theorem}{Theorem}{Theorems}
\crefname{proposition}{Proposition}{Propositions}
\crefname{lemma}{Lemma}{Lemmas}
\crefname{corollary}{Corollary}{Corollaries}
\crefname{remark}{Remark}{Remarks}
\crefname{example}{Example}{Examples}
\crefname{definition}{Definition}{Definitions}
\crefname{problem}{Problem}{Problems}
\crefname{observation}{Observation}{Observation}
\crefname{construction}{Construction}{Construction}

\theoremstyle{theorem}
\newtheorem{innercustomgeneric}{\customgenericname}
\providecommand{\customgenericname}{}
\newcommand{\newcustomtheorem}[2]{%
  \newenvironment{#1}[1]
  {%
   \renewcommand\customgenericname{#2}%
   \renewcommand\theinnercustomgeneric{##1}%
   \innercustomgeneric
  }
  {\endinnercustomgeneric}
}
\newcustomtheorem{theoremX}{Theorem}
\newcustomtheorem{lemmaX}{Lemma}
\newcustomtheorem{conjectureX}{Conjecture}

\DeclareMathOperator{\sign}{sign}

\DeclareMathOperator{\conv}{conv}

\DeclareMathOperator{\Span}{span}
\DeclareMathOperator{\im}{im}

\DeclareMathOperator{\tr}{tr}

\DeclareMathOperator{\Id}{Id}

\DeclareMathOperator{\vol}{vol}  	
\DeclareMathOperator{\Int}{int}

\newcommand{\ball}[3]{B_{#2}^{#1}\kern-1pt(#3)}

\let\<=\langle
\let\>=\rangle
\let\x=\times

\def\...{...}
\newcommand{\shortStyle}{\textit}
\newcommand{\ie}{\shortStyle{i.e.,}}

\newcommand{\eg}{\shortStyle{e.g.}}

\newcommand{\wrt}{\shortStyle{w.r.t.}}

\newcommand{\cf}{\shortStyle{cf.}}

\newcommand{\resp}{resp.}

\let\angle=\measuredangle

\makeatletter
\renewcommand*{\eqref}[1]{%
  \hyperref[{#1}]{\textup{\tagform@{\ref*{#1}}}}%
}
\makeatother

\newcommand{\tempnewpage}{}

\newcommand{\mcolor}{OrangeRed}
\newcommand{\msays}[1]{{\footnotesize\textcolor{\mcolor}{\textbf{M:} #1}}}
\newcommand{\WOmega}{\Omega^{\text{\normalfont\textsc w}}}
\newcommand{\Wstress}{\omega^{\text{\normalfont\textsc w}}}
\newcommand{\bsWstress}{{\bs\omega}^{\text{\normalfont\textsc w}}}

\DeclareMathOperator{\real}{\textsc{real}}

\newcommand{\hh}[2]{h_{\sigma_{#1}}^{\sigma_{#2}}}
\newcommand{\hhdot}[2]{{\dot h}_{\sigma_{#1}}^{\sigma_{#2}}}
\newcommand{\HH}[2]{{\vec h}_{\sigma_{#1}}^{\sigma_{#2}}}

\newcommand{\nn}[2]{n_{\sigma_{#1}}^{\sigma_{#2}}}

\renewcommand{\S}{\mathbb{S}}
\renewcommand{\H}{\mathbb{H}}

\newcommand{\triplestack}[3]{
    \kern-1pt
    \begin{smallmatrix*}[l]
        #1
        \\[-0.2ex]
        #2
        \\[-0.2ex]
        #3
    \end{smallmatrix*}
}
\newcommand{\triplestacktiny}[3]{
    \kern-1pt
    \begin{smallmatrix*}[l]
        \scriptscriptstyle #1
        \\[-0.3ex]
        \scriptscriptstyle #2
        \\[-0.3ex]
        \scriptscriptstyle #3
    \end{smallmatrix*}
}

\makeatletter
\newcommand{\addresseshere}{%
  \enddoc@text\let\enddoc@text\relax
}
\makeatother

\begin{document}


\title%
[Second-order rigidity of coned polytope frameworks]%
{Second-order rigidity of coned polytope frameworks and the stress-flex conjecture from a vector-valued Schläfli formula} 
		
\author[E.\ Pachyli]{Eleni Pachyli}
\address{{TU Wien, Wiedner Hauptstrasse 8-10/104 1040 Vienna, Austria}}
\email{eleni.pachyli@tuwien.ac.at}

\author[R.\ Prosanov]{Roman Prosanov}
\address{{Autonomous University of Barcelona, Department of Mathematics, Building C Science Faculty, 08193 Bellaterra, Spain}}
\email{roman.prosanov@univie.ac.at}

\author[M.\ Winter]{Martin Winter}
\address{Max-Planck Institute for Mathematics in the Sciences, Inselstraße 22, 04103 Leipzig, Gernany}
\email{martin.winter@mis.mpg.de}
	

		
\date{\today}
\begin{abstract}
A \emph{coned polytope framework} (CPF) is the bar-joint framework~obtained from the 1-skeleton of a convex polytope by coning over some interior point.
It was recently shown that CPFs are rigid, though the exact order of rigidity remained open.
In this paper we introduce the Wachspress stress and use it to show that CPFs are prestress stable, in particular, second-order rigid.
To this end, we resolve the stress-flex conjecture in the case of the Wachspress stress by identifying its dual formulation as a corollary of a vector-valued Schläfli-type formula introduced by Schlenker and Souam.
We give a new and purely discrete-geometric proof of this generalized Schläfli formula.
\end{abstract}

\vspace*{-2em}
\maketitle




\vspace*{-1ex}
\section{Introduction}


A \emph{coned polytope framework} (or \emph{CPF}) is a bar-joint framework constructed from the 1-skeleton of a convex polytope $P\subset\RR^d$ by inserting additional bars 
between the vertices of $P$ and some interior point $p_\star\in\Int(P)$.
It was recently proven~that CPFs are rigid \cite[Theorem 4.5]{winter2024rigidity}. 
This means that any sufficiently small perturbation of the 1-skeleton of $P$~that~preserves the edge lengths and the distances between the vertices and $p_\star$, must be congruent to $P$.
This result is surprising~since polytope~ske\-leta can be quite sparse.
It is also somewhat unsatisfying as the proof~gives~no~indication of the actual order of rigidity.
In fact, many CPFs are not first-order rigid, which means, they admit differentiable deformations that preserve the bar lengths at least in first order. 
Such a first-order flex is shown in \cref{fig:pointed_polytope_flex}.
%
The analogous question for second-order rigidity remained open \cite[Question 5.2.]{winter2024rigidity}.
In this article we show that second-order rigidity can be derived from a generalization of the Schläfli formula applied to the dual polytope.

\vspace*{1ex}
\begin{figure}[h!]
    \centering
    \includegraphics[width=0.30\linewidth]{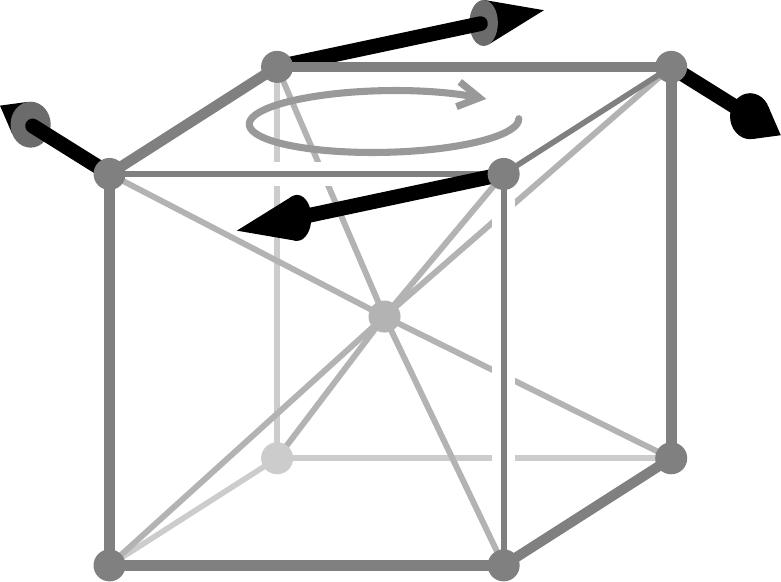}
    \caption{A coned polytope framework of the standard cube together with a first-order flex that (infinitesimally) twists the top square.}
    \label{fig:pointed_polytope_flex}
\end{figure}

%
\begin{theorem}
    \label{res:prestress_stable}
    Coned polytope frameworks are prestress stable, and hence, second-order rigid.
\end{theorem}

\iftrue

    In fact, we obtain a tensegrity version of this result: a CPF is still prestress~stable, even if the polytope edges are modeled as \textit{cables} (can become shorter but~not longer) and the cone edges are modeled as \emph{struts} (can become longer but not~shorter).

\else

    In fact, we obtain a tensegrity version of this result: a CPF is still prestress~stable, even if the polytopes edges are modeled as \textit{struts} (can become longer but~not~shorter) and the cone edges are modeled as \emph{cables} (can become shorter but not longer).

\fi

\textit{Second-order rigidity} and \textit{prestress stability} were initially introduced by Connelly and Whiteley for tensegrity frameworks \cite{connelly1996second}, and are now employed more generally for the analysis of geometric constraint systems that are not rigid in the first order.
Notable applications are the resolution of Roth's conjecture for convex bar-cable polygons \cite{connelly1996second}, the rigidity of arbitrarily triangulated convex surfaces \cite{connelly1980rigidity, connelly2017prestress}, and the rigidity of exceptional polytopal point-hyperplane frameworks \cite{winter2026secondorder}.

Prestress stability is hereby the stronger and, for us, more convenient notion.\nls
It builds on the existence of a distinguished stress for the framework that ``blocks''~all first-order flexes (details given in \cref{sec:second_order}).
For proving \cref{res:prestress_stable} we show that the polytopal structure of a CPF, no matter how sparse in edges, gives rise to such a special stress.
This stress is known as the \emph{Wachspress stress} $\bsWstress$, a detailed account of which is given in \cref{sec:WI_stress}.

%


For proving prestress stability using $\bsWstress$ we follow the path laid out in an earlier note.
It was shown in \cite{winter2024stressflex} that prestress stability would follow from a curious property of the Wachspress stress called \emph{stress-flex orthogonality} (\cf\ \eqref{eq:stress_flex_ortho} below).
We recall~this connection in \cref{sec:stress_flex_to_CPF}.
The question for whether $\bsWstress$ has this~proper\-ty~became known as the \emph{weak stress-flex conjecture} \cite[Conjecture 5.1]{winter2024stressflex}.
We resolve this conjecture affirmatively:
\begin{theorem} 
\label{res:weak_stress_flex}
If $\bsdot p$ is a first-order motion of a CPF with $\dot p_\star=0$, then
\begin{equation}
    \label{eq:stress_flex_ortho}
    \sum_{i\not=\star} \Wstress_{\star i} \dot p_i = 0.
\end{equation}
Here the sum is over the vertices of $P$.
\end{theorem}
%

The central observation that lead to the resolution of the weak stress-flex conjec\-ture is that the special structure of the Wachspress stress makes \eqref{eq:stress_flex_ortho} dual to~a~corollary of a generalized Schläfli formula first proven by Schlenker and Souam \cite{SS2} (the details of this connection are given in \cref{sec:SS}):

\begin{theorem}[{\cite[Theorem 1.11]{SS2}}]
\label{res:SS}
\label{res:SS*}
\label{res:SS**}
Given a differentiable 1-parameter family $P^t\subset \RR^d$ of polytopes, at $t=0$ it holds
\[\sum_{\sigma_{d-1}} \dot V_{\sigma_{d-1}}n_{\sigma_{d-1}}^{\sigma_d}=\sum_{\sigma_{d-2}}\dot\theta_{\sigma_{d-2}}^{\sigma_d} V_{\sigma_{d-2}}b_{\sigma_{d-2}}.\]
%
Here, the sums are over the $(d-1)$- and $(d-2)$-dimensional faces $\sigma_{d-1},\sigma_{d-2}\subset P$ respectively, $V_{\sigma_{d-1}}$ and $V_{\sigma_{d-2}}$ denote their volumes, $n_{\sigma_{d-1}}^{\sigma_d}$ denotes the normal vector of $\sigma_{d-1}$, and
$\theta_{\sigma_{d-2}}^{\sigma_d}$ and $b_{\sigma_{d-2}}$ denote the exterior dihedral angle and barycenter of $\sigma_{d-2}$ respectively.
\end{theorem}

More details on the origin of the statement and the notation are given in \cref{sec:schlafli}.
Schlenker and Souam prove \cref{res:SS**} using a double limiting~argument from an analogous result for smooth bodies in spherical (and hyperbolic)~space. 
First, they approximate the (spherical) polytope by smooth convex bodies.
Subsequently they approximate the Euclidean case by diverging the radius of the sphere to infinity.
In \cref{sec:Stoker_type}, we present a new, purely discrete-geometric proof of \cref{res:SS}, avoiding approximation arguments.

Lastly, note that \Cref{res:weak_stress_flex} resolves only a special case of the \emph{(strong) stress-flex conjecture} -- an analogous conjecture about general stresses and piecewise linear surfaces. 
We briefly discuss the general conjecture and approaches in \cref{sec:outlook}.

\subsection*{Acknowledgements} We are very grateful to Ivan Izmestiev, Steven Gortler, Louis Theran, Sean Dewer and Robert Connelly for fruitful discussions. 

\subsection*{Funding}
Roman Prosanov is funded in whole by the Austrian Science Fund (FWF) \url{https://doi.org/10.55776/J4955}.

Martin Winter is funded by the SPP 2458 ``Combinatorial Synergies'' (project ID 539851419), funded by the Deutsche Forschungsgemeinschaft (DFG, German Research Foundation).

For open access purposes, the authors have applied a CC BY public copyright license to any author-accepted manuscript version arising
from this submission.

\tempnewpage

\section{Polytopes, frameworks and rigidity}
\label{sec:basic_notions}

Throughout the text let $P\subset\RR^d$ denote a (convex) polytope spanned by its~fini\-tely many vertices $p_1,...,p_n$.
Let $p_\star\in \Int(P)$ be some interior point of $P$.

By $G_P=(V_P,E_P)$ we denote the \Def{vertex-edge graph} of $P$, that is, $V_P=\{1,...,n\}$, with $ij\in E_P$ if and only of $p_i$ and $p_j$ span an edge of $P$.
%
%
Then $$G_P^\star=(V_P^\star,E_P^\star)=(V_P\cupdot \{\star\}, E_P \cupdot E_\star)$$
denotes the \Def{coned vertex-edge graph}, that is, the graph obtained from $G_P$~by~adding a new \Def{cone vertex} $\star\in V_P^\star$ adjacent to all vertices in $V_P$ along new \Def{cone edges} $E_\star$.\nls
For distinction, the edges in $E_P$ we shall call \Def{polytope edges}.

A $d$-dimensional \Def{bar-joint framework} (or \Def{framework} for short) is a pair $(G,\bs p)$ consisting of a graph $G=(V,E)$ and an embedding map $\bs p\: V\to\RR^d$.
One thinks of the edges of $G$ as straight line segments between the embedded vertices.
For~example, we can treat the polytope's \Def{1-skeleton} as a framework of the vertex-edge graph $G_P$.
It is common to use the terms \Def{joint} and \Def{bar} for the embeddings of vertices and edges to distinguish from their combinatorial meaning within the graph $G$.
%

Given $P$ and $p_\star$ as above, the corresponding \Def{coned polytope framework} (or \Def{CPF} for short) is the framework $(G_P^\star,\bs p)$ of the coned vertex-edge graph that embeds the polytope's vertices and the cone point in the obvious way.

For the rest of this section we recall the necessary aspects of rigidity theory with some comments on how this applies to the theory of CPFs.

\subsection{Motions and flexes}


Rigidity theory studies deformations of frameworks~that preserve bar lengths.
Formally, a \Def{motion} of $\bs p$ is a continuous 1-parameter family~$\bs p^t\!$, $t\in[0,1]$, of embeddings with $\bs p^0 = \bs p$ and $\|p_i^t-p_j^t\|$ constant throughout for all edges $ij\in E$.
%
%
%
A motion is \Def{trivial} if it preserves \emph{all} pairwise distances between~vertices, not only along edges.
Trivial motions are induced by continuous families of global isometries, such as translations and rotations.
A non-trivial motion is called a \Def{flex}.
A framework is \Def{flexible} if it has a flex, and is said to be \Def{rigid} otherwise.

Using the tools of Wachspress Geometry it was proven in \cite[Theorem 4.5]{winter2024rigidity}~that CPFs are rigid.
This is surprising since polytope skeleta can be~quite sparse (see~\cref{res:sparse} below).
Rigidity theory has tools to quantify~this~discrepancy by investigating the order at which rigidity sets in.
For CPFs this leads to the open questions we study in this article.
In the next two section we recall the first- and second-order theory of framework rigidity.

\subsection{First-order theory}
\label{sec:first_order}

First-order theory studies deformations of a framework that preserves bar lengths in the first order, that is, for $ij\in E$
\begin{equation}
    \label{eq:first_order_d}
    \tfrac{\mathrm d}{\mathrm dt}\|p_i^t-p_j^t\|\big|_{t=0} = 0
    \;\;
    \implies
    \;\;
    0=\tfrac{\mathrm d}{\mathrm dt}\|p_i-p_j\|^2
    = 
    \big\<p_i-p_j, \tfrac{\mathrm d}{\mathrm dt}p_i-\tfrac{\mathrm d}{\mathrm dt}p_j\big\>.
\end{equation}
For convenience we drop the superscript $t$ if we evaluate at $t=0$.


This motivates the following definition: a map $\bsdot p:V\to\RR^d$ is a \Def{first-order motion} if it satisfies
\begin{equation}
    \label{eq:first_order_flex}
    \<p_j-p_i,\dot p_j-\dot p_i\>=0,\quad\text{for all $ij\in E$}.
\end{equation}
Every differentiable family as in \eqref{eq:first_order_d}, in particular every differentiable family of motions, gives rise to such a first-order motion via $\bsdot p := \tfrac{\mathrm d}{\mathrm dt}\bs p^t|_{t=0}$.
The converse~is true as well: every first-order motion comes from differentiating a family as in \eqref{eq:first_order_d}, though not necessarily from differentiating a motion.

It is possible to characterize the first-order motions that are obtained from differ\-entiating trivial motions. They are precisely of the form
\begin{equation}
    \label{eq:trivial_flex}
    \bsdot p = S\bs p + t,    
\end{equation}
where $S\in\RR^{d\x d}$ is \emph{skew-symmetric} and $t\in \RR^d$. Such first-order motions are called \Def{trivial} and they exist for every framework irrespective of structural details.
If non-trivial, they are called \Def{first-order flexes}.
A framework is \Def{first-order flexible} if it has first-order flexes, and \Def{first-order rigid} otherwise.

A major benefit of first-order theory is that it reduces rigidity analysis to linear algebra and thereby provides practical necessary criteria for the existence of a flex: if a framework is first-order rigid, then it is rigid \cite{asimow1978rigidity}.
It moreover allows us to make heuristic statements about the expected rigidity behavior of a system.
For example, the following argument suggests that many CPFs should not be rigid:

\begin{remark}
    \label{res:sparse}
    Let $P\subset\RR^d$ be a \emph{simple} polytope, that is, each vertex has degree~$d$.
    We perform the naive count ``degrees of freedom (DOF) minus constraints'' on its coned framework.
    First, ignoring the bar constraints, each joint has $d$ independent ways to move.
    This yields $d(|V_P|+1)$ DOFs. 
    Heuristically, we expect to lose exactly one DOF from each bar constraint.
    There are $|E_P|$ bars from the polytope edges, and $|V_P|$ bars from the cone edges.
    For a $d$-regular graph we have $2|E_P|=d|V_P|$.
    The total count therefore gives
    $$\text{\#DOFs}-\text{\#constraints} = d(|V_P|+1)-(|V_P|+|E_P|)=|V_P|(\tfrac d2-1)+d.$$
    In dimension $d$ there are $d$ dimensions of translations and ${d\choose 2}$ dimensions of rotations, contributing $d+{d \choose 2}$~unavoidable freedoms.
    A straightforward computation shows that for $d\ge 3$ and~$|V_P|\ge \smash{\frac{d(d-1)}{d-2}\ge 6}$ the above count exceeds the number of trivial motions and therefore~suggests that the CPF is first-order flexible.

    For example, if $P$ is the 3-dimensional cube (which has $|V_P|=8$), we expect to find a single first-order flex. An exact computation confirms this (see also \cref{fig:pointed_polytope_flex}).
\end{remark}

\subsection{Second-order theory}
\label{sec:second_order}

Second-order theory studies continuous deformations that preserve bar length up to second order.
It is employed for the analysis~of~frameworks that fail the first-order rigidity tests, and hence a natural choice for the study of CPFs.
Here we use an alternative (but equivalent) formulation of the theory developed by Connelly and Whiteley \cite{connelly1996second} that expresses second-order rigidity using~the bilinear pairing between first-order flexes and stresses. 

An \Def{equilibrium stress} (or \Def{stress} for short) of a framework $(G,\bs p)$ is a map $\bs\omega\: E\to\RR$ that satisfies the \Def{stress equilibrium condition} at each vertex:
$$\sum_{\mathclap{j:ij\in E(G)}} \omega_{ij} (p_j-p_i)=0,\quad\text{for all $i\in V$}.$$
%
%
The framework is \Def{second-order rigid} if for each first-order flex $\bsdot p$ there is a stress~$\bs\omega$ that \Def{blocks} $\bsdot p$, which means
\begin{equation}
    \label{eq:blocks}
    Q_{\bs \omega}(\bsdot p):=
    \sum_{ij\in E}\omega_{ij}\|\dot p_i-\dot p_j\|^2>0.
\end{equation}
One can show that $Q_{\bs \omega}(\bsdot p)=0$ for every trivial first-order motion $\bsdot p$, \ie\ no trivial motion can be blocked by a stress.

Establishing second-order rigidity can be difficult since the blocking stress might be different for each first-order flex.
The following is a stronger but more convenient notion: a framework is \Def{prestress stable} if there is a \textit{single} stress $\bs\omega$ that~blocks~\emph{all}~first-order flexes.
%
%
The notion and name originate in structural engineering: if a physical framework is constructed with a stress built into the structure, then any deformation in the direction of a first-order flex increases the energy expressed by \eqref{eq:blocks},\nls \ie\ is energetically unfavorable.
It holds
$$
\text{first-order rigid}
\;\Rightarrow\;
\text{prestress stable}
\;\Rightarrow\;
\text{second-order rigid}
\;\Rightarrow\;
\text{rigid}.
$$
None of the inverse implications holds in general \cite{connelly1996second}.

Lastly, we mention an alternative way of expressing \eqref{eq:blocks}:
recall that the \Def{stress matrix} $\Omega\in\RR^{V\x V}$ of a stress $\bs\omega$ has entries
$$\Omega_{ij} := \begin{cases}
    \omega_{ij} & \text{if $ij\in E$}
    \\[0.5ex]
    -\sum\limits_{\mathclap{k:ik\in E}} \omega_{ik} & \text{if $i=j$}
    \\[1.8ex]
    0 & \text{otherwise}
\end{cases}.$$
In particular, $\Omega$ is symmetric and has row and column sum zero.
%
If we interpret $\bsdot p$ as a $(d\x V)$-matrix, one can check that 
\begin{equation}
    \label{eq:prestress_stable_Omega}
    \tr(\dotbs p \kern1pt\Omega\kern1pt \dotbs p\T)
        = \sum_{ij\in E}\omega_{ij}\|\dot p_i-\dot p_j\|^2
        = Q_{\bs\omega}(\bsdot p).
\end{equation}

\tempnewpage

\section{The Wachspress stress}
\label{sec:WI_stress}

To certify prestress stability of a coned polytope framework we need to choose a suitable stress.
Stresses are an expression of redundancies in frameworks; and since a CPF can be quite sparse (\cf\ \cref{res:sparse}), one might not expect to find a stress at all.
In this section we demonstrate how the piecewise linear boundary structure of a polytope gives rise to such a distinguished stress -- the \emph{Wachspress~stress}.

As before, let $P\subset\RR^d$ be a convex polytope with interior point $p_\star\in\Int(P)$.
Let $(G_P^\star,\bs p)$ be the corresponding coned polytope framework.
By
$$(P-p_\star)^\circ := \{y\in \RR^d\mid  \<y,x - p_\star\> \le 1\text{ for all $x\in P$}\}$$
we denote the polar dual of $P$, obtained by polarizing at $p_\star$.
Let $F_i^\diamond$ be the facet of the polar dual that is dual to the vertex $p_i$ of $P$.
Likewise, let $F_{ij}^\diamond$ be the ridge (\ie\ codimension-2 face) of the polar dual that is dual to the edge $ij$ of $P$. 
Finally, let $C_i:=F_i^\diamond\vee p_\star$ and $C_{ij} := F_{ij}^\diamond \vee p_\star$ be the cones over $F_i^\diamond$ and $F_{ij}^\diamond$ respectively, where $\vee$ expresses the convex hull with the cone point. 

The \Def{Wachspress stress} $\bsWstress$ (also \Def{Izmestiev stress}) of $(G,\bs p)$ is given by
\begin{equation}\label{eq:Wachspress_stress}
\Wstress_{\star i} := \frac{\vol(F_i^\diamond)}{\|p_i - p_\star\|}\;\; \text{(for $i\not=\star$)}
\qquad 
\Wstress_{ij} := -\frac{\vol(C_{ij})}{\|p_j-p_i\|}
    \;\;\text{(for $i,j\not=\star$)}
\end{equation}

That this is indeed a stress of the framework follows from a repeated application of the \Def{Minkowski balancing property}:

\begin{theorem}[{Minkowski balancing property, see \eg\ \cite[Section 8.2.1]{schneider2013convex}}] \label{res:Minkowski_balancing}
    Let $P\subset \RR^d$ be a polytope.
    For a facet $\sigma_{d-1}$ of $P$, let $n_{\sigma_{d-1}}^{\sigma_{d}}\in\RR^d$ be the (outwards~pointing) unit normal and let $V_{\sigma_{d-1}}$ be its volume.
    Then
    \begin{equation}
        \label{eq:Minkowski_balancing}
        \sum_{\sigma_{d-1}} n_{\sigma_{d-1}}^{\sigma_{d}} V_{\sigma_{d-1}} = 0,
    \end{equation}
    Here the sum is over the facets $\sigma_{d-1}$ of $P$.
\end{theorem}

\begin{lemma}
    The Wachspress stress $\bsWstress$ is a stress of the coned polytope framework. 
\end{lemma}
\begin{proof}
    

    The stress equilibrium at the cone vertex $\star$ is obtained by applying Minkow\-ski's balancing property to $(P-p_\star)^\circ$. 
    Note that the unit normal vector $n_i$ of facet $F_i^\diamond\subset (P-p_\star)^\circ$ points in the direction $p_i-p_\star$:
    $$\sum_{i\not=\star} \omega_{\star i} (p_i-p_\star) = \sum_{i\not=\star} \frac{\vol(F_i^\diamond)}{\|p_i-p_\star\|} (p_i-p_\star)= \sum_{i\not=\star} n_i \vol(F_i^\diamond) = 0.$$
    %

    The stress equilibrium at a non-cone vertex $i\not=\star$ is obtained by applying~Min\-kow\-ski's balancing property to the cone $C_i\subset (P-p_\star)^\circ$.
    The facets of $C_i$ are~$F_i^\diamond$~and the cones $C_{ij}$ for $ij\in E_P$ (see \cref{fig:cube_cones}). 
    Note that the normal vector $n_{ij}$ at the facet $C_{ij}$ points in the direction of $p_j-p_i$.
    With this, we obtain
    \begin{align*} 
        \sum_{\mathclap{j:ij\in E_P^\star}} \omega_{ij} (p_j-p_i)
        &\;=\;
        \omega_{\star i}(p_\star-p_i)  \;+\; \sum_{\mathclap{j:ij\in E_P}} \omega_{ij}(p_j-p_i)
        \\[0ex]&\;=\;
        -\frac{\vol(F_i^\diamond)}{\|p_i-p_\star\|} (p_i-p_\star) \;-\;\sum_{\mathclap{j:ij\in E_P}} \;\;\frac{\vol(C_{ij})}{\|p_j-p_i\|} (p_j-p_i) 
        \\[0.2ex]&\;=\; 
        -\Bigg(n_i\vol(F_i^\diamond)\;+\;\sum_{\mathclap{j:ij\in E_P}}  n_{ij} \vol(C_{ij})\Bigg)
        = 0.\qedhere
    \end{align*}
    %
\end{proof}

\begin{figure}[h!]
    \centering
    \includegraphics[width=0.6\linewidth]{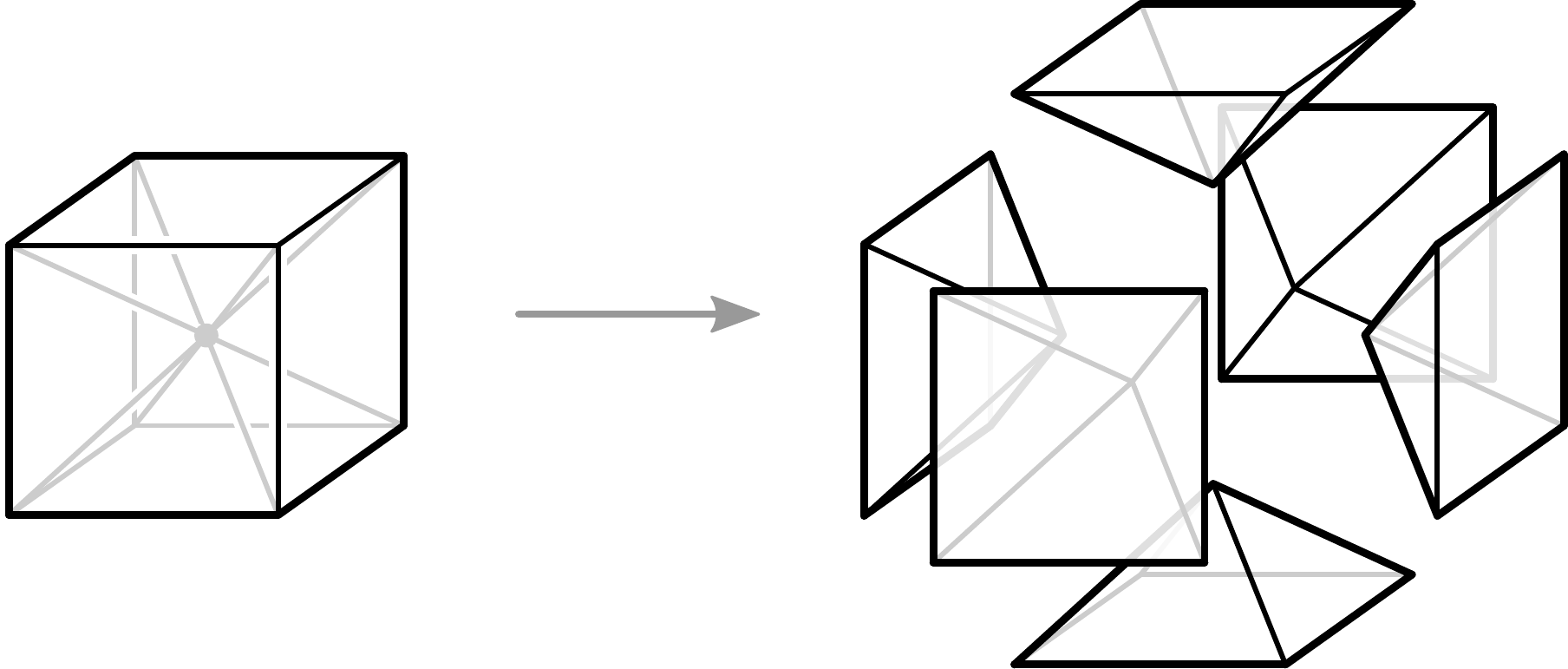}
    \caption{Decomposition of the dual polytope $(P-p_\star)^\circ$ into its facial pyramids. The entries of the Wachspress stress are derived from facet volumes in this decomposition. 
    }
    \label{fig:cube_cones}
\end{figure}


The name \emph{Wachspress stress} derives from the fact that the entries $\Wstress_{\star i}$ on the~cone edges are (up to a normalization factor) precisely the \emph{Wachspress coordinates} $\alpha\in\RR^n$ of the point $p_\star$ in $P$.
Wachspress coordinates are a form of generalized barycentric coordinates initially introduced in geometric modeling by Wachspress for polygons and simple 3-polytopes \cite{wachspress1975rational}, and subsequently generalized by Warren to polytopes and polyhedra of general combinatorics and dimension \cite{warren1996barycentric}.
They have since been studied in a variety of context, including algebraic geometry \cite{kohn2020projective,irving2014geometry} and polyhedral rigidity \cite{winter2024rigidity}.

Similarly, the stress entries $\smash{\Wstress_{ij}}$ on the polytope edges correspond to the entries of the \emph{Izemstiev matrix} $M\in\RR^{n\times n}$ for the point $p_\star$ in $P$.
The Izmestiev matrix~is a Colin de Verdière matrix of the polytope's edge graph. 
It was introduced~by~Izmestiev in \cite{izmestiev2010colin} as a generalization of construction by Lovász for dimension $d=3$~\cite{lovasz2001steinitz}.
This construction has since found applications in polyhedral rigidity \cite{winter2024rigidity} and~combinatorics \cite{narayanan2021spectral}.

A detailed explanation for the relation between the Wachspress coordinates and the Izmestiev matrix is given in \cite[Section 3.1 -- 3.3]{winter2024rigidity}.
We recall the essential facts.


\begin{theorem}[{\cite[Theorem 3.3]{winter2024rigidity}}]
\label{res:Izmestiev}
Given a polytope $P\subset\RR^d$ with $n$ vertices and with interior point~$p_\star\in\Int(P)$,
there exists a symmetric matrix $M\in\RR^{n\x n}$ (the \Def{Izmestiev matrix}~of~$p_\star$~in~$P$) with the following properties:
\begin{myenumerate}
    \item $M_{ij}>0$ whenever $ij\in E_P$,
    \item $M_{ij}=0$ whenever $i\not=j$ and $ij\not\in E_P$,
    \item $M$ has a unique positive eigenvalue (\ie\ of multiplicity one), and
    \item $\ker M = \Span\kern.5pt(p_1 - p_\star,..., p_n-p_\star)\T$.
\end{myenumerate}
\end{theorem}

Explicit expressions are known for the entries of the Izmestiev matrix.
The~off-diagonal ``on-edge'' entries are typically given in the form \cite[Equation (3.4)]{winter2024rigidity}:
\begin{equation}
    \label{eq:Izemstiev_entries}
    M_{ij} = \frac{\vol(F_{ij}^\diamond)}{\|p_i-p_\star\|\|p_j-p_\star\|\sin\angle(p_i-p_\star,p_j-p_\star)},\quad\text{if $ij\in E_P$}.
\end{equation}
The diagonal entries can be computed from the following relation to the (unnorma\-lized) \Def{Wachspress coordinates} $\tilde\alpha_i := \vol(F_i^\diamond)/\|p_i-p_\star\|=\omega_{\star i}$:

\begin{theorem}[{\cite[Corollary 3.6]{winter2024rigidity}}]
\label{res:Izemstiev_diagonal}
$\displaystyle\sum_j M_{ij} = \tilde\alpha_i$ for all $i\in\{1,...,n\}$.
\end{theorem}


We can now express the \Def{stress matrix $\WOmega$} of $\Wstress$ in terms of these quantities.


\begin{lemma}
\label{res:block_form}
$\WOmega$ has the block form
$$\WOmega 
\,:=\!\!\!\!\!\!
\kbordermatrix{
      & 1 & \cdots & n & \omit & \star \cr
 \phantom1    & & & & \omit\vrule & \rule{0pt}{2.2ex} \cr
 \phantom\vdots & & \mathllap- M & & \omit\vrule & \tilde\alpha \cr
 \phantom n    & & & & \omit\vrule & \rule{0pt}{2.2ex} \cr
 \omit & \multispan{5}\hrulefill \cr
 \phantom\star & & \overset{\phantom.}{\tilde\alpha\T} & & \omit\vrule & -v \cr
}
\mathrlap{\color{white}
\kbordermatrix{
      & \cr
 \textcolor{black}1    & \rule{0pt}{2.2ex} \cr
  \textcolor{black}{\text{\scriptsize$\vdots$}} & -\hat M & \cr
  \textcolor{black}n    & \rule{0pt}{2.2ex} \cr
 \omit & \hrulefill \cr
  \textcolor{black}\star & \overset{\phantom.}{\tilde\alpha\T}  \cr
}}\quad,
$$
where $M$ is the Izmestiev matrix of $p_\star$ in $P$, $\tilde\alpha$ are the (unnormalized) Wachspress coordinates of $p_\star$ in $P$, and $v:=\vol(P-p_\star)^\circ$.
\end{lemma}
\begin{proof}
    The identity holds for the entries in the $\tilde\alpha$-blocks by definition.

    We next check the $M$-block. We first verify that $M_{ij}$ from \eqref{eq:Izemstiev_entries} agrees with~$-\Wstress_{ij}$ from \eqref{eq:Wachspress_stress} whenever $ij\in E_P$. This amounts to checking
    \begin{equation}
    \label{eq:target}
    \displaystyle -\Wstress_{ij}=\frac{\vol(C_{ij})}{\|p_j-p_i\|}\overset?=\frac{\vol(F_{ij}^\diamond)}{\|p_i-p_\star\|\|p_j-p_\star\|\sin\angle(p_i-p_\star,p_j-p_\star)}=M_{ij}.
    \end{equation}
    The argument is elementary geometric:
    let $h_{ij}$ be the altitude of the triangle~$\triangle p_\star p_i p_j$ at the vertex $p_\star$, \ie\ $h_{ij}$ is the distance of $p_\star$ from the affine span of the edge~$e_{ij}:=\conv\{p_i,p_j\} \subset P$.
    Since $F_{ij}^\diamond\subset P^\circ$ is dual to the edge $e_{ij}$, the distance of its affine span from $p_\star$ is $\smash{h_{ij}^{-1}}$.
    In particular, $\vol(C_{ij})=\vol(F_{ij}^\diamond)\smash{h_{ij}^{-1}}$.
    After substituting this into the target equation \eqref{eq:target} and comparing both sides, it remains to verify
    $$h_{ij} = \frac{\|p_i-p_\star\|\|p_j-p_\star\|\sin\angle(p_i-p_\star,p_j-p_\star)}{\|p_j-p_i\|}.$$
    But the right side is a well-known expression for the altitude in a triangle.

    The identity on the diagonal entries of the $M$-block now follow from \cref{res:Izemstiev_diagonal} and the fact that stress matrices have zero row sum.
    

    It remains to verify the bottom-right entry.
    Note that the distance of (the affine span of) the facet $F_i^\diamond$ from $p_\star$ is $\|p_i-p_\star\|^{-1}$.
    Hence $\vol(C_i)=\vol(F_i^\diamond)\|p_i-p_\star\|^{-1}=\tilde\alpha_i$.
    Since $\WOmega$ has zero row sums, the bottom-right entry computes to
    $$v = \sum_i \tilde\alpha_i = \sum_i \frac{\vol(F_i^\diamond)}{\|p_i-p_\star\|} = \sum_i \vol(C_i) = \vol(P-p_\star)^\circ,$$
    where for the last identity we used that the cones $C_i$ have disjoint interior and cover $(P-p_\star)^\circ$ (see \cref{fig:cube_cones}).
\end{proof}

The properties of $\WOmega$ are essentially determined by its upper-left block.
For later use we prove the following:

\begin{lemma}
    \label{res:Omega_prop}\,
    \begin{myenumerate}
        \item $\WOmega$ has a unique negative eigenvalue.
        \item if $\WOmega\bs q\T=0$ then $\bs q=A\bs p+t$ for some $A\in\RR^{d\x d}$ and $t\in\RR^d$.\footnote{By convention $A \bs p+t$ stands for $A\bs p+t\bs 1\T=(Ap_1+t,...,Ap_n+t)$, where $\bs 1=(1,...,1)\in\RR^n$ is the all-ones vector.}
    \end{myenumerate}
\end{lemma}
{\color{black}
\begin{proof}
    Let $\bs 0=(0,...,0)\in\RR^n$ and $\bs 1=(1,...,1)\in\RR^n$ be the all-zeroes and all-ones vectors respectively.
    Then
    \begin{equation}
    \label{eq:matrix_blocks}
    \WOmega 
    = 
    \begin{bmatrix}
    \;\,-M & \;\;M\bs 1 \\
    \bs 1\T\! M & \!\!-\bs 1\T\! M \bs 1
    \end{bmatrix}
    = 
    S\T
    \begin{bmatrix}
    -M & \bs 0 \\
    \phantom-\bs 0\T & 0
    \end{bmatrix}
    S,
    \quad
    \text{with }
    S = 
    \begin{bmatrix}
    \Id & -\bs 1 \\
    \bs 0\T & \bs 1
    \end{bmatrix}
    \end{equation}
    By Sylvester's law of inertia, $\WOmega$ and $-M$ have in particular the same number~of~ne\-gative eigenvalues. Then \itm1 follows from \cref{res:Izmestiev} \itm3. 

    To prove \itm2, we split $\bs q=(\bs q_P, q_\star)$ into the polytope and cone parts.    
    Writing out $\WOmega \bs q\T\! = 0$ for the upper blocks (\cf\ \eqref{eq:matrix_blocks}), we get $M(\bs q_P\T- q_\star\T) = 0$.
    \cref{res:Izmestiev} \itm4 in particular gives us
    $$\bs q_P-q_\star = A(p_1-p_\star,...,p_n-p_\star) = A(\bs p_P-p_\star),$$
    which rearranges to $\bs q_P= A \bs p_P + (q_\star - Ap_\star)$.
    If we set $t:=q_\star-Ap_\star$ we get both
    $$q_\star = Ap_\star + t,\quad \text{and} \quad \bs q_P = A\bs p_P+t,$$
    proving \itm2.
\end{proof}
}

\tempnewpage

\section{The stress-flex conjecture}
\label{sec:stress_flex}

It was shown in \cite{winter2024stressflex} that prestress stability of CPFs would follow from the stress-flex orthogonality for the Wachspress stress.
It is now understood that this~is,\nls most likely, not specific to the Wachspress stress or convex polytopes, but a general~property of the bilinear pairing between stresses and first-order flexes in coned 1-skeleta of closed PL surfaces. 
We comment on this generalized \emph{strong stress-flex conjecture} in \cref{sec:outlook}.
In this section we recall how the weak stress-flex conjecture for the Wachspress stress implies prestress stability.
We then derive the weak stress-flex conjecture from a vector-valued Schläfli formula due to Schlenker and Souam \cite{SS2}.

\subsection{From the weak stress-flex conjecture to prestress stability}
\label{sec:stress_flex_to_CPF}

We recall the form of the stress-flex conjecture resolved here:

\begin{theoremX}{\ref{res:weak_stress_flex}}
Let $(G_P^\star,\bs p)$ be a coned polytope framework with Wachspress stress $\bsWstress$.
If $\bsdot p$ is a first-order motion of $(G_P^\star,\bs p)$ with $\dot p_\star=0$, then
$$\sum_{i\not=\star} \Wstress_{\star i} \kern.5pt \dot p_i = 0.$$
\end{theoremX}

Below we provide a self-contained proof for how prestress stability of a CPF~follows from \cref{res:weak_stress_flex}. 
The argument is in large parts from \cite{winter2024stressflex}:

\begin{theoremX}{\ref{res:prestress_stable}}
    Coned polytope frameworks are prestress stable.
\end{theoremX}
\begin{proof}
    Let $\WOmega$ be the stress matrix of the Wachspress stress, and $\bsdot p$ some first-order motion with $\dot p_\star=0$.
    Following \eqref{eq:prestress_stable_Omega} it~suffices to prove \itm1 that $\tr(\bsdot p\kern0.5pt\WOmega\bsdot p\T)\ge 0$~and \itm2 that $\tr(\bsdot p\kern0.5pt\WOmega\bsdot p\T)=0$ only if $\bsdot p$ is trivial.
    
    %
    

    For \itm1 let $\bs e_\star\in\RR^{V_P^\star}$ be the standard unit vector with 1-entry at the cone~vertex~$\star$. 
    Then the following two identities hold:
    %
    %
    %
    %
    %
    \begin{align}
    \label{eq:temp5}
    \bs e_\star \T\WOmega\bs e_\star&=\WOmega_{\star\star}\overset{\mathclap{\ref{res:block_form}}}=-\vol(P^\circ-p_\star)<0 \notag,
    \\
    \bs e_\star \T\WOmega \bsdot p\T
        &= \sum_{i,j} \WOmega_{ij} e_{\star i} \dot p_j
        = \sum_j \WOmega_{\star j}\dot p_j 
        = \sum_{i\not=\star} \Wstress_{\star i}\dot p_i
        \overset{\ref{res:weak_stress_flex}}= 0.
    \end{align}
    Note that it is the last equality where we use stress-flex orthogonality.
    
    Let now $(\bsdot p)_k\in\RR^{V_P^\star}$ denote the $k$-th row of $\bsdot p$. 
    %
    %
    Then \eqref{eq:temp5} states that $\bs e_\star$ and $(\bsdot p)_k$ are orthogonal \wrt\ the quadratic form induced by $\WOmega$\!.
    If \mbox{$(\dotbs p)_k\WOmega (\dotbs p)_k\T<0$},\nls then~$\bs e_\star$ and $(\dotbs p)_k$ would span a 2-dimensional subspace of $\RR^V\!$ on which $\WOmega$~is~negative~definite.
    But by \cref{res:Omega_prop} \itm1 $\WOmega$ has at most one negative eigenvalue.
    This would be a contradiction.
    Hence
    \begin{equation}
        \label{eq:temp6}
        (\dotbs p)_k\WOmega(\dotbs p)_k\T\ge 0
        \;\;\implies\;\;
        \tr(\bsdot p\kern1pt\WOmega \bsdot p\T) = \mathrlap{\phantom{\sum}}\smash{\sum_{k}} (\dotbs p)_k\WOmega(\dotbs p)_k\T \ge 0,
    \end{equation}
    which proves \itm1.

    For \itm2, suppose that $\bsdot p\kern1pt\WOmega \bsdot p\T\!\!=0$.
    We need to show that $\bsdot p$ is trivial.
    In \cite{winter2024stressflex}~this is argued using affine flexes and ruled surfaces.
    The subsequent argument~is~a~refor\-mulation thereof that requires no knowledge of these concepts.

    First, using \eqref{eq:temp6} we can conclude $(\dotbs p)_k\WOmega(\dotbs p)_k\T = 0$ for all $k$.
    Since $\WOmega$ is positive semi-definite on the span of all $(\bsdot p)_k$, we also conclude $\WOmega(\bsdot p)_k\T=0$.
    \Cref{res:Omega_prop}~\itm2 then gives us $\bsdot p = A\bs p + t$ for some matrix $A\in\RR^{d\x d}$ and vector $t\in\RR^d$.
    The~first-order flex condition \eqref{eq:first_order_flex} becomes
    \begin{equation}
        \label{eq:vanishes_on_E}
        0 = \<p_i-p_j, \dot p_i-\dot p_j\> = (p_i-p_j)\T \!\!A\kern1pt (p_i-p_j),\quad\text{for all $ij\in E_P^\star$}.
    \end{equation}
    %
    %
    In other words, the quadratic form $Q(x):=x\T\! A x$ vanishes on all edge directions~of $(G_P^\star,\bs p)$.
    It now suffices to show $Q\equiv 0$: this is equivalent to $A$ being skew-symmetric and $\bsdot p=A\bs p+t$ being trivial according to \eqref{eq:trivial_flex}.

    In the remainder of the proof we show that for each face $\sigma\subset P$, $Q$ vanishes on $\mathcal E_\sigma:=\Span(p_i-p_\star\mid p_i\in \sigma)$.
    The claim $Q\equiv 0$ then follows from $\mathcal E_\sigma=\RR^d$ whenever $\sigma$ is a facet of $P$.
    We proceed by induction on the dimension of the face $\sigma$. 

    \emph{Case $\dim(\sigma)=0$:} the face $\sigma$ is a vertex $p_i$, and $Q$ vanishes on $p_i-p_\star$ by \eqref{eq:vanishes_on_E}.
    
    \emph{Case $\dim(\sigma)=1$:} the face $\sigma$ is some edge between vertices $p_i$ and $p_j$. 
    The vectors $p_i-p_\star, p_j-p_\star$ and $p_i-p_j$ are three coplanar edge directions of $(G_P^\star,\bs p)$ (their span is $\mathcal E_\sigma$), and since $p_\star\in\Int(P)$, any two are linearly independent.
    The restriction $Q|_{\mathcal E_\sigma}$ has degree at most two, but vanishes on three different 1-dimensional subspaces. Hence $Q$ vanishes on all of $\mathcal E_\sigma$.

    \emph{Case $\dim(\sigma)\ge 2$:} then there are at least three $(\delta-1)$-dimensional faces $\tau_1,\tau_2,\tau_3\subset\sigma$. The restriction $Q|_{\mathcal E_\sigma}$ has degree at most two but, by induction hypothesis, vanishes on three distinct $(\delta-1)$-dimensional subspace $\mathcal E_{\tau_i}$ $\subset \mathcal E_{\sigma}$. Hence $Q$ vanishes on all of $\mathcal E_\sigma$.
    \qedhere
\end{proof}

\begin{remark}
    The Wachspress stress has the property that it is strictly positive on the cone edges and strictly negative on the polytope edges.
    Hence, the proof of \cref{res:prestress_stable} actually shows something stronger: a CPF is not only prestress stable as a bar-joint framework, but is also prestress stable as a \emph{tensegrity framework} with cables for the polytope edges and struts for the cone edges. 
    In contrast to bars, which must stay of a fixed length during a motion, struts are allowed to get longer (but not shorter), and cables are allowed to get shorter (but not longer).
    This has also been mentioned in \cite{winter2024stressflex}.
    For the second-order theory of tensegrities, see \cite{connelly1996second}.
\end{remark}




\subsection{The vector-valued Schläfli formula and a proof of the stress-flex conjecture}
\label{sec:SS}

In this part we show that, in the case of the Wachspress stress, the stress-flex conjecture is a special case of a generalized Schläfli formula proven by Schlenker and Souam.
We first recall the essential notions.

%
%
%
From here on let $P^t$ be a 1-parameter family of polytopes, all of which have the same number of facets $F_1,...,F_n$, though can otherwise be of different combinatorics.
We say that $P^t$ is continuous (\resp\ differentiable) if both the facet normals $n_i^t$ and facet volumes $V_i^t$ change continuously (resp.\ differentiably) in $t$.
As before, we drop the superscript $t$ if we evaluate at $t=0$.
We write $\dot n_i$ and $\dot V_i$ to denote the derivative of $n_i^t$ and $V_i^t$ at $t=0$ respectively.

We recall the \Def{vector-valued Schläfli formula} by Schlenker and Souam \cite{SS2}:

%

\begin{theoremX}{\ref{res:SS*}}
\begin{equation}
\label{mains}
\sum_{\sigma_{d-1}} \dot V_{\sigma_{d-1}}n_{\sigma_{d-1}}^{\sigma_d}=\sum_{\sigma_{d-2}}\dot\theta_{\sigma_{d-2}}^{\sigma_d} V_{\sigma_{d-2}}b_{\sigma_{d-2}}.
\end{equation}
%
\end{theoremX}

The relevant special case of \cref{res:SS} can be motivated naturally.
Consider the first variation of the Minkowski balancing property (\cf\ \cref{res:Minkowski_balancing}):
\begin{equation}
    \label{eq:Minkowski_impl}
    \frac{\mathrm d}{\mathrm dt} \sum_i n_i V_i = 0 \;\;\implies\;\; \sum_i \dot n_i V_i + \sum_i n_i \dot V_i=0
\end{equation}
It turns out that if the dihedral angles of $P^t$ do not change in first order, then both summands on the left side of \eqref{eq:Minkowski_impl} vanish individually:  
\begin{corollary}
\label{res:Stoker_type}
If $\dot\theta_{ij}=0$ for all facets $F_i$ and $F_j$ that are adjacent at $t=0$, then
$$\sum_{i} n_i \dot V_i = \sum_i \dot n_i V_i = 0.$$
\end{corollary}
\begin{proof}
    Set $\dot\theta_{\sigma_2}=0$ in \cref{res:SS}.
\end{proof}

%
A new self-contained proof of \cref{res:SS} is given in \cref{sec:Stoker_type}.
The reader only interested in \cref{res:Stoker_type} can focus on \cref{sec:preparations} and \cref{sec:proof_of_stoker_type}.

It remains to show that the stress-flex conjecture for $\bsWstress$ can be seen as a dual formulation of \cref{res:Stoker_type}, and therefore proven using it.

\begin{theoremX}{\ref{res:weak_stress_flex}}
Let $(G_P^\star,\bs p)$ be a coned polytope framework with Wachspress stress $\bsWstress$.
If $\bsdot p$ is a first-order motion of $(G_P^\star,\bs p)$ with $\dot p_\star=0$, then
$$\sum_{i\not=\star} \Wstress_{\star i} \kern.5pt \dot p_i = 0.$$
\end{theoremX}

\begin{proof}
    By translation we may also assume $p_\star=0$.
    Recall that there is a differentiable 1-parameter family $\bs p^t$ of embeddings with $\bs p^0 = \bs p$ and $\dotbs p=\tfrac{\mathrm d}{\mathrm dt} \bs p^t|_{t=0}$.
    Since $\dotbs p$ is a first-order flex, it preserves bar lengths in first-order. That is
    \begin{align}
        \label{eq:temp15}
        0 &= \tfrac{\mathrm d}{\mathrm dt} \|p_i - p_\star\|^2 = \tfrac{\mathrm d}{\mathrm dt}\|p_i\|^2.
        \\
        \label{eq:temp2}
        0 &= \tfrac{\mathrm d}{\mathrm dt} \|p_i - p_j\|^2 = \tfrac{\mathrm d}{\mathrm dt}(\|p_i\|^2 - 2\<p_i,p_j\> +\|p_j\|^2) \overset{\eqref{eq:temp15}}= - 2\tfrac{\mathrm d}{\mathrm dt} \<p_i,p_j\> .
    \end{align}

    Now consider the following 1-parameter family $P^t\subset\RR^d$ of polytopes: 
    $$P^t:= \big\{x\in\RR^d\mid \<p_i^t, x\> \le 1\text{ for all $i\in\{1,...,n\}$}\big\}.$$
    For $t=0$ this is precisely the polar dual $P^\circ$.
    The number of facets of $P^t$ stays constant within a neighborhood of $t=0$, and facet normals $n_i^t$ and volumes $V_i^t$ are differentiable in $t$. 
    In particular, $\dot n_i=\dot p_i / \|p_i\|$ because $\|p_i\|$ is constant in first order by \eqref{eq:temp15}.
    Moreover, the dihedral angles $\theta_{ij}=\arccos\<n_i,n_j\>=\arccos(\<p_i,p_j\>/\|p_i\|\|p_j\|)$ are also constant in first-order by \eqref{eq:temp15} and \eqref{eq:temp2}.

    Then for the Wachspress stress $\bsWstress$ holds
    \begin{equation}
        \label{eq:temp1}
        \sum_i  \Wstress_{\star i} \dot p_i = \sum_i \frac{\vol(F_i)}{\|p_i\|} \dot p_i = \sum_i \vol(F_i) \frac{\dot p_i}{\|p_i\|}=\sum_i V_i \dot n_i \,\overset{\mathclap{\ref{res:Stoker_type}}}=\, 0,
    \end{equation}
    where in the last equality we apply \cref{res:Stoker_type}.
    %
\end{proof}



\tempnewpage

\section{A proof of the generalized Schläfli formula}
\label{sec:Stoker_type}



In this section we give a direct discrete-geometric proof of \cref{res:SS**}. This section benefits from an adjusted notation, since in the proofs here we need to refer to all parts of the face lattice of a polytope. First, let us provide a bit of history of the Schl\"afli formula and its generalizations.

\subsection{The Schläfli formula and its relatives}
\label{sec:schlafli}

As earlier, $P^t \subset \RR^d$, $t \in [0,1]$, $P^0=P$, is a differentiable family of convex polytopes with the same number of facets. 
In this section $\sigma_k$ denotes a face of $P$ of dimension $k$. Note that $\sigma_d$ is the polytope itself.
For faces $\sigma_k \subset \sigma_{k+1}$ of $P$, ${n}_{\sigma_{k}}^{\sigma_{k+1}}$ is the (outwards pointing) unit normal vector from $\sigma_{k+1}$ towards ${\sigma_{k}}$. For faces $\sigma_k \subset \sigma_{k+2}$, $\theta_{\sigma_{k}}^{\sigma_{k+2}}$ is the exterior dihedral angle of $\sigma_{k}$ in $\sigma_{k+2}$. For a face $\sigma$, $V_\sigma$ is its volume and $b_\sigma$ is its barycenter.  
%

The Euclidean version of the celebrated \Def{Schläfli formula} states that

\begin{theorem}[{\cite[Chapter 7.2.2]{AVS},~\cite{Mil}}]
\label{res:schlafli}
\[\sum_{\sigma_{d-2}}\dot\theta_{\sigma_{d-2}}^{\sigma_d} V_{\sigma_{d-2}}=0.\]
%
%
%
\end{theorem}
%
%


The Schläfli formula extends to polytopes in the spherical $d$-space $\S^d$ and in the hyperbolic $d$-space $\H^d$, where in the right-hand side zero is replaced by $\pm\dot V_{\sigma_d}$ respectively. The Schläfli formula is known for its broad spectrum of applications, ranging from the study of isometric embeddings and their rigidity~\cite{BI, Pro}, to discrete conformality~\cite{BPS, IPW, Riv} and to geometrization results~\cite{BLP, Luo}, to name a few. It was generalized in a plethora of directions. As we already discussed, the relevant one for our paper is its vector-valued extension due to Schlenker and Souam~\cite{SS2}, which we now re-state for convenience of the reader:


\begin{theoremX}{\ref{res:SS*}}
\begin{equation}
\label{mains}
\sum_{\sigma_{d-1}} \dot V_{\sigma_{d-1}}n_{\sigma_{d-1}}^{\sigma_d}=\sum_{\sigma_{d-2}}\dot\theta_{\sigma_{d-2}}^{\sigma_d} V_{\sigma_{d-2}}b_{\sigma_{d-2}}.
\end{equation}
%
\end{theoremX}

Note that the left-hand side in \cref{mains} is translation-invariant, from which the classical Schläfli formula can be derived immediately.
%
%
The left-hand side of \cref{mains} also appears in the first-order variations of the Minkowski balancing property,~\cref{res:Minkowski_balancing}.





Similarly to the classical Schläfli formula, in~\cite{SS2} Schlenker and Souam also supply the spherical and hyperbolic versions of \cref{res:SS*}. As we mentioned in the introduction, they actually deduce \cref{res:SS*} from its spherical counterpart by a blow-up argument, and prove the spherical and hyperbolic results using an approximation of polytopes by smooth convex bodies.
Here we provide a different approach to \cref{res:SS**}.
Our proof is based on the orthoscheme decomposition, also introduced by Schläfli. We will employ the mentioned classical theorems of Schläfli and Minkowski, as well as another result of Schlenker and Souam from~\cite{SS2}, for which we will also provide a new proof, and which we formulate now.

Let $n_1, \ldots, n_k$ be an ordered $k$-tuple of points on $\S^2 \subset \RR^3$ with $n_i \neq \pm n_{i-1}$, $\alpha_i$ be the oriented angle $n_{i-1}n_in_{i+1}$ (the angle between spherical segments) with $\alpha_i \not\in\{ 0,\pm\pi\}$, $i\in\{1, ..., k\}$. Let $m_i$ be the dual point~to~the oriented segment $n_in_{i+1}$ and $\beta_i$ be the oriented angle $m_{i-1}m_i m_{i+1}$ (hence, $\beta_i$ is equal to the oriented length of $n_i n_{i+1}$). Consider differentiable 1-parameter families $n_i^t$, $n_i^0=n_i$, and the differentiable families $m_i^t, \alpha_i^t$ and $\beta_i^t$ determined by them. A \Def{first-order deformation} of such a configuration consists of the tangent vectors to these families at $t=0$. As usually, we suppress $t=0$ for first-order deformations. Note that the data of such a first-order deformation is uniquely determined by $\dot n_i$.

\begin{theorem}[{\cite[Theorem 1.1]{SS2}}]
\label{res:gluck}
For a
first-order deformation of a configuration as above we get
\[\sum_i \dot\alpha_i n_i =\sum_i\dot\beta_i m_i.\]
\end{theorem}

%

The case of \cref{res:gluck} corresponding to the isometric deformations, \ie\ when all $\dot\beta_i=0$, is quite frequent in the polyhedral literature, see, \eg\, \cite[Section 10.1]{Ale},~\cite[Section 3.8]{Glu},~\cite[Theorem $\text{A}_{\text{S}}$]{Sch}. As we already mentioned, the full case is proven in~\cite{SS2} by Schlenker and Souam. However, it does not have a separate proof there, the authors rather deduce it from their very general results. In particular, their proof requires an approximation of polyhedral curves by smooth curves and a limiting argument. In 
%
%
\cref{sec:gluck} we will give a streamlined, self-contained proof of \cref{res:gluck}, generalizing the approach of 
%
Gluck from~\cite{Glu}. A different proof of~\cref{res:gluck} implicitly follows from~\cite{GBKKRS}, although the statement is not explicitly formulated there and~\cite{GBKKRS} uses a very different language from ours. It is interesting that in~\cite{GBKKRS} \cref{res:gluck} is used in the context of classification of constant mean curvature surfaces in Euclidean 3-space.

\subsection{Proof preparations}
\label{sec:preparations}

We have a convex polytope $P \subset \RR^d$. 
\cref{res:gluck} implies 

\begin{corollary}
Pick two faces $\sigma_{k} \subset \sigma_{k+3}$ of $P$. We have
\begin{equation}
\label{glucklink}
\sum_{\substack{\sigma_{k+2}:\\\sigma_k \subset \sigma_{k+2}\subset \sigma_{k+3}}} \dot \theta_{\sigma_k}^{\sigma_{k+2}} n_{\sigma_{k+2}}^{\sigma_{k+3}} = \sum_{\substack{\sigma_{k+1}:\\\sigma_k \subset \sigma_{k+1}\subset \sigma_{k+3}}} \dot \theta_{\sigma_{k+1}}^{\sigma_{k+3}} n_{\sigma_k}^{\sigma_{k+1}}.
\end{equation}
\end{corollary}

\begin{proof}
This follows from applying \cref{res:gluck} to the spherical link of $\sigma_k$ in $\sigma_{k+3}$.
\end{proof}

For a face $\sigma$, let $o_\sigma$ be the orthogonal projection from the origin $0 \in \RR^d$ to the affine span of $\sigma$. If necessary, we slightly perturb the coordinate system so that for any faces $\sigma \neq \tau$, $o_{\sigma}$ does not coincide with $o_{\tau}$. For faces $\sigma \subset \tau$, the \Def{orthonumber} 
%
%
$h^\tau_\sigma$ is the distance from $o_\sigma$ to $o_\tau$. However, when $\sigma$ has codimension one in $\tau$, we consider $h^\tau_\sigma$ as the oriented distance, \ie\ taken with the minus sign if $o_\tau$ and $\tau$ belong to the different sides of the affine span of $\sigma$ in the affine span of $\tau$. See \cref{fig:point_positions} for an example of our notation for two triples of faces $\sigma_{k} \subset \sigma_{k+1} \subset \sigma_{k+2}$ and $\sigma_{k}\subset \sigma_{k+1}^* \subset \sigma_{k+2}$.


\vspace*{1ex}
\begin{figure}[h!]
    \centering
    \includegraphics[width=0.63\linewidth]{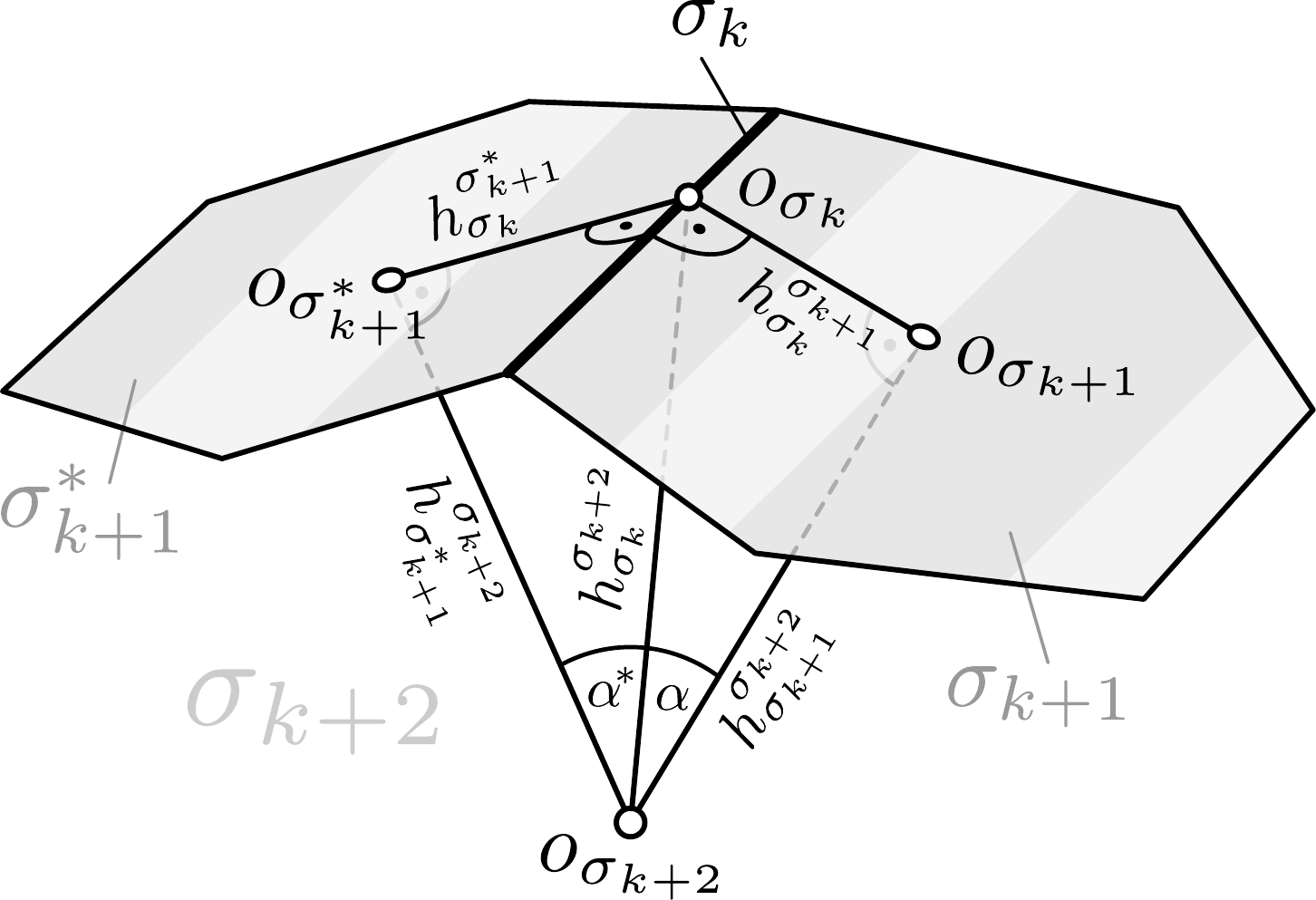}
    \caption{Visualization of the setup for the proof of \cref{res:SS}.}
    \label{fig:point_positions}
\end{figure}

Consider three faces $\sigma_{k} \subset \sigma_{k+1} \subset \sigma_{k+2}$. Denote by $\alpha$ the angle $\angle o_{\sigma_{k+1}} o_{\sigma_{k+2}} o_{\sigma_{k}} $, taken with the minus sign if $h_{\sigma_k}^{\sigma_{k+1}}$ and $h_{\sigma_{k+1}}^{\sigma_{k+2}}$ have opposite signs. See \cref{fig:point_positions}.
Note that since $o_{k+1}$ does not coincide with $o_{k}$ and with $o_{k+2}$, we get either $0<\alpha<\pi/2$ or $-\pi/2<\alpha<0$. We have \[ h_{\sigma_{k}}^{\sigma_{k+1}} = h_{\sigma_{k+1}}^{\sigma_{k+2}} \tan \alpha.  \]
Differentiating it, we obtain
\[ \dot h_{\sigma_{k}}^{\sigma_{k+1}}  = \dot h_{\sigma_{k+1}}^{\sigma_{k+2}}\tan\alpha + h_{\sigma_{k+1}}^{\sigma_{k+2}} \frac{\dot{\alpha}}{\cos^2 \alpha}. \]
We have \[\cos^2 \alpha =\left(\frac{h_{\sigma_{k+1}}^{\sigma_{k+2}}}{h_{\sigma_{k}}^{\sigma_{k+2}}}\right)^2.\]
From the last two equations, we get
\[\dot h_{\sigma_{k}}^{\sigma_{k+1}} = \dot h_{\sigma_{k+1}}^{\sigma_{k+2}}  \tan \alpha + \frac{\left(h_{\sigma_{k}} ^{\sigma_{k+2}}\right)^2}{h_{\sigma_{k+1}}^{\sigma_{k+2}} } \dot{\alpha}.\]
Multiplying by $h_{\sigma_{k+1}}^{\sigma_{k+2}}$ and using $h_{\sigma_{k+1}}^{\sigma_{k+2}}\tan \alpha =h_{\sigma_{k}}^{\sigma_{k+1}}$ yields
\begin{equation}\label{eq:trianglerelations}
    \dot h_{\sigma_{k}}^{\sigma_{k+1}}h_{\sigma_{k+1}}^{\sigma_{k+2}}  = h_{\sigma_{k}}^{\sigma_{k+1}} \dot h_{\sigma_{k+1}}^{\sigma_{k+2}} + \left(h_{\sigma_{k}} ^{\sigma_{k+2}}\right)^2 \dot{\alpha}.
\end{equation}

Suppose first that both $h_{\sigma_k}^{\sigma_{k+1}}$ and $h_{\sigma_{k+1}}^{\sigma_{k+2}}$ are positive. We differentiate the Pythagorean theorem for the triangle $o_{\sigma_{k+1}} o_{\sigma_{k+2}} o_{\sigma_{k}}$ and obtain
%
%
\[\dot h_{\sigma_{k}}^{\sigma_{k+1}}h_{\sigma_{k}}^{\sigma_{k+1}} +\dot h_{\sigma_{k+1}}^{\sigma_{k+2}}h_{\sigma_{k+1}}^{\sigma_{k+2}}= \dot h_{\sigma_{k}} ^{\sigma_{k+2}}h_{\sigma_{k}} ^{\sigma_{k+2}},\]
\[
\implies\quad \dot h_{\sigma_{k}}^{\sigma_{k+1}}\sin\alpha +\dot h_{\sigma_{k+1}}^{\sigma_{k+2}}\cos\alpha= \dot h_{\sigma_{k}} ^{\sigma_{k+2}}.\]

Let $n$ be the unit normal vector from $o_{\sigma_{k+2}}$ towards $o_{\sigma_k}$ and $n^\perp$ be the unit normal vector from the triangle $o_{\sigma_{k+1}} o_{\sigma_{k+2}} o_{\sigma_{k}}$ towards $o_{\sigma_{k}} o_{\sigma_{k+2}}$. From the last two equations, we compute
%
\begin{align}
\label{semikite}
\dot\alpha h_{\sigma_k}^{\sigma_{k+2}}n
&=(\dot h_{\sigma_{k}}^{\sigma_{k+1}}\cos\alpha- \dot h_{\sigma_{k+1}}^{\sigma_{k+2}}\sin\alpha)(n_{\sigma_k}^{\sigma_{k+1}}\sin\alpha+n_{\sigma_{k+1}}^{\sigma_{k+2}}\cos\alpha)
\\ \notag
&=\dot h_{\sigma_{k}}^{\sigma_{k+1}}n_{\sigma_{k+1}}^{\sigma_{k+2}}-\dot h_{\sigma_{k+1}}^{\sigma_{k+2}}n_{\sigma_{k}}^{\sigma_{k+1}}
\\ \notag
&\qquad +(\dot h_{\sigma_{k}}^{\sigma_{k+1}}\sin\alpha+\dot h_{\sigma_{k+1}}^{\sigma_{k+2}}\cos\alpha) (n_{\sigma_k}^{\sigma_{k+1}}\cos\alpha-n_{\sigma_{k+1}}^{\sigma_{k+2}}\sin\alpha)
\\ \notag
&=\dot h_{\sigma_{k}}^{\sigma_{k+1}}n_{\sigma_{k+1}}^{\sigma_{k+2}}-\dot h_{\sigma_{k+1}}^{\sigma_{k+2}}n_{\sigma_{k}}^{\sigma_{k+1}}+\dot h_{\sigma_{k}} ^{\sigma_{k+2}}n^\perp.
\end{align}
One can check that for the other possible signs of $h_{\sigma_k}^{\sigma_{k+1}}$ and $h_{\sigma_{k+1}}^{\sigma_{k+2}}$, for the same left-hand side of~\eqref{semikite}, the right-hand side either remains the same or becomes $\dot h_{\sigma_{k}}^{\sigma_{k+1}}n_{\sigma_{k+1}}^{\sigma_{k+2}}-\dot h_{\sigma_{k+1}}^{\sigma_{k+2}}n_{\sigma_{k}}^{\sigma_{k+1}}-\dot h_{\sigma_{k}} ^{\sigma_{k+2}}n^\perp.$

Now pick two faces $\sigma_{k}\subset \sigma_{k+2}$. Note that then there exist exactly two $(k+1)$-faces $\sigma_{k+1}$ and $\sigma_{k+1}^\ast$ such that $\sigma_{k}\subset\sigma_{k+1}\subset \sigma_{k+2}$ and $\sigma_{k}\subset \sigma_{k+1}^\ast \subset \sigma_{k+2}$. Define $\alpha$ as above, and similarly define $\alpha^\ast$ for the triangle $o_{\sigma_{k+1}^\ast} o_{\sigma_{k+2}} o_{\sigma_{k}}$. See \cref{fig:point_positions}. Note that either both pairs of numbers $h_{\sigma_k}^{\sigma_{k+1}}$, $h_{\sigma_{k+1}}^{\sigma_{k+2}}$ and $h_{\sigma_k}^{\sigma_{k+1}^\ast}$, $h_{\sigma_{k+1}^\ast}^{\sigma_{k+2}}$ have the same sign or they both have the opposite signs. This means that both $\alpha$ and $\alpha^\ast$ have the same sign. Observe that, depending on the sign, either \[\alpha+\alpha^\ast=\theta_{\sigma_k}^{\sigma_{k+2}}\qquad\text{or}\qquad -\alpha-\alpha^\ast=\pi-\theta_{\sigma_k}^{\sigma_{k+2}}.\] In either case,
\begin{equation}
\label{alphatheta}
\dot\alpha+\dot\alpha^\ast=\dot\theta_{\sigma_k}^{\sigma_{k+2}}.
\end{equation}

Let $n$ be the unit normal vector from $o_{\sigma_{k+2}}$ towards $o_{\sigma_k}$ and $n^\perp$ be the unit normal vector from the triangle $o_{\sigma_{k+1}} o_{\sigma_{k+2}} o_{\sigma_{k}}$ towards $o_{\sigma_{k}} o_{\sigma_{k+2}}$ and $n^{\perp\ast}$ be the unit normal vector from the triangle $o_{\sigma_{k+1}^\ast} o_{\sigma_{k+2}} o_{\sigma_{k}}$ towards $o_{\sigma_{k}} o_{\sigma_{k+2}}$. We have $n^\perp=-n^{\perp\ast}$. Note also that the term $\dot h_{\sigma_{k}} ^{\sigma_{k+2}}n^\perp$ appears in~\eqref{semikite} for the triangle $o_{\sigma_{k+1}} o_{\sigma_{k+2}} o_{\sigma_{k}}$ with the same sign as the term $\dot h_{\sigma_{k}} ^{\sigma_{k+2}}n^{\perp\ast}$ appears in such equality for the triangle $o_{\sigma_{k+1}^\ast} o_{\sigma_{k+2}} o_{\sigma_{k}}$.
From this, \eqref{alphatheta} and~\eqref{semikite}, we deduce
\begin{multline}
\label{kite}
\dot\theta_{\sigma_k}^{\sigma_{k+2}}h_{\sigma_k}^{\sigma_{k+2}}n_{\sigma_k}^{\sigma_{k+2}}=(\dot\alpha+\dot\alpha^\ast)h_{\sigma_k}^{\sigma_{k+2}}n_{\sigma_k}^{\sigma_{k+2}}\\
=\dot h_{\sigma_{k}}^{\sigma_{k+1}}n_{\sigma_{k+1}}^{\sigma_{k+2}}-\dot h_{\sigma_{k+1}}^{\sigma_{k+2}}n_{\sigma_{k}}^{\sigma_{k+1}}+\dot h_{\sigma_{k}}^{\sigma_{k+1}^\ast}n_{\sigma_{k+1}^\ast}^{\sigma_{k+2}}-\dot h_{\sigma_{k+1}^\ast}^{\sigma_{k+2}}n_{\sigma_{k}}^{\sigma_{k+1}^\ast}.
\end{multline}

%
%
A \Def{flag} is an increasing sequence of faces of $P$. We say that a flag is \Def{full} if it has a face in every dimension, that it is \Def{partial} otherwise, and that it is \Def{$k$-order} if it has a face in each dimension up to and including $k$.

Let $f$ be a full flag. We denote its respective elements by $f_0,\dots,f_d$, according to dimension. We write for short $h_{f_k}:= h_{f_k} ^{f_{k+1}}$ and ${n}_{f_k}:={n}_{f_k}^{f_{k+1}}$. Denote by $\alpha_{f_k}$ the angle $\angle o_{f_{k-1}} o_{f_{k}} o_{f_{k-2}}$, taken with the minus sign if $h_{f_{k-1}}$ and $h_{f_{k-2}}$ have opposite signs. Denote by $V_f$ the oriented volume of $f$, \ie\ $V_f:=\frac{1}{d!}h_{f_0}\ldots h_{f_{d-1}}$. 

We denote a $k$-order flag by $f^k$ and denote its elements by $f^k_0, \ldots, f^k_k$. Our conventions for full flags apply to flags of any order. 

Given a face $\sigma_k$, its volume $V_{\sigma_k}$ is
\begin{equation}
\label{volume}V_{\sigma_k}=\displaystyle\sum_{f^k: f^k_k=\sigma_k} V_{f^k}=\dfrac{1}{k!}  \displaystyle\sum_{f^k: f^k_k=\sigma_k} h_{f^k_0}\ldots h_{f^k_{k-1}}.
\end{equation}
In particular, 
\[ 
    V_{\sigma_{d-1}}= \dfrac{1}{(d-1)!} \sum_{f: f_{d-1}= \sigma_{d-1}} h_{f_0}  h_{f_1}  \ldots h_{f_{d-2}}.
\]
A variation of this gives
\begin{equation}\label{eq:varvol}
\dot{V}_{\sigma_{d-1}}=\dfrac{1}{(d-1)!}  \sum_{f:f_{d-1}=\sigma_{d-1}}\sum_{k=0}^{d-2} h_{f_0} \ldots \dot{h}_{f_k}  \ldots h_{f_{d-2}}.
\end{equation}

For a full flag $f$ and $k$ with $0\leq k\leq d-3$, equation~\eqref{eq:trianglerelations} yields
\[  \dot{h}_{f_{k}}h_{f_{k+1}}={h}_{f_{k}}\dot  h_{f_{k+1}} + \left(h_{f_{k}} ^{f_{k+2}}\right)^2 \dot{\alpha}_{f_{k+2}}.\]
Multiplying this equation by the other 
orthonumbers, we obtain
\begin{align*}
h_{f_0} \ldots h_{f_{k-1}} \dot{h}_{f_{k}} &h_{f_{k+1}} \ldots h_{f_{d-2}} 
\\&=h_{f_0} \ldots h_{f_{k}} \dot{h}_{f_{k+1}} h_{f_{k+2}}  \ldots  h_{f_{d-2}}  
\\&\qquad + h_{f_0} \ldots h_{f_{k-1}} \left(h_{f_{k}} ^{f_{k+2}}\right)^2 \dot{\alpha} _{f_{k+2}}  h_{f_{k+2}}  \ldots h_{f_{d-2}}.  
\end{align*}
Using it in~\eqref{eq:varvol}, we get
\begin{multline}
\label{eq:first}
\dot{V}_{\sigma_{d-1}}=\frac{1}{(d-2)!}\sum_{f:f_{d-1}=\sigma_{d-1}}h_{f_0}\ldots h_{f_{d-3}}\dot h_{f_{d-2}}
\\+\dfrac{1}{(d-1)!}  \sum_{f:f_{d-1}=\sigma_{d-1}}\sum_{k=0}^{d-3} (k+1) h_{f_0}  \ldots h_{f_{k-1}} \left(h_{f_{k}} ^{f_{k+2}}\right)^2 \dot{\alpha} _{f_{k+2}} h_{f_{k+2}}  \ldots h_{f_{d-2}}.
\end{multline}
Note that we of course mean that in the second sum for $k=0$ the corresponding term does not have the part $h_{f_0}\ldots h_{f_{k-1}}$ and for $k=d-3$ the corresponding term does not have the part $h_{f_{k+2}}  \ldots h_{f_{d-2}}$.

\subsection{Proof of \cref{res:Stoker_type}}
\label{sec:proof_of_stoker_type}

%
Here we prove \cref{res:Stoker_type}, which is of primary importance to our proof of the weak stress-flex conjecture, as explained in \cref{sec:stress_flex}. In our current language it states that if all $\dot\theta_{\sigma_{d-2}}^{\sigma_d}=0$, then
\begin{equation}
\label{Stoker_type_eq}
\sum_{\sigma_{d-1}} \dot V_{\sigma_{d-1}}n_{\sigma_{d-1}}^{\sigma_d}=0.
\end{equation}
We will use the assumption that all $\dot\theta_{\sigma_{d-2}}^{\sigma_d}=0$ only in the very end of our proof, as we will use our computations from this subsection in the proof of the general case of \cref{res:SS} in the next subsection. 

%
Substitute~\eqref{eq:first} into $\sum_{\sigma_{d-1}} \dot V_{\sigma_{d-1}}n_{\sigma_{d-1}}^{\sigma_d}$ and obtain
\begin{multline*}
\sum_{\sigma_{d-1}} \dot V_{\sigma_{d-1}}n_{\sigma_{d-1}}^{\sigma_d}=\sum_{\sigma_{d-2}\subset \sigma_{d-1}}V_{\sigma_{d-2}}\dot h_{\sigma_{d-2}}^{\sigma_{d-1}} n_{\sigma_{d-1}}^{\sigma_{d}} 
\\+\dfrac{1}{(d-1)!} \sum_f\sum_{k=0}^{d-3} (k+1)   h_{f_0} \ldots h_{f_{k-1}} \left(h_{f_{k}} ^{f_{k+2}}\right)^2  \dot{\alpha}_{f_{k+2}} h_{f_{k+2}}  \ldots h_{f_{d-2}} {n}_{f_{d-1}} .
\end{multline*}

Denote the summands
\[T_1:=\sum_{\sigma_{d-2}\subset \sigma_{d-1}}V_{\sigma_{d-2}}\dot h_{\sigma_{d-2}}^{\sigma_{d-1}} n_{\sigma_{d-1}}^{\sigma_{d}},\]
\[T_2:=\dfrac{1}{(d-1)!}\sum_f\sum_{k=0}^{d-3} (k+1)   h_{f_0} \ldots h_{f_{k-1}} \left(h_{f_{k}} ^{f_{k+2}}\right)^2  \dot{\alpha}_{f_{k+2}} h_{f_{k+2}}  \ldots h_{f_{d-2}} {n}_{f_{d-1}} .\]

Let us deal with $T_1$. Given any $\sigma_{d-2}$, there exist exactly two facets $\sigma_{d-1}$ and $\sigma_{d-1}^*$ such that $\sigma_{d-2} \subset \sigma_{d-1}$ and $\sigma_{d-2} \subset \sigma_{d-1}^*$.
%
%
Pairing the summands this way and using~\eqref{kite}, we obtain 
\begin{equation*}
T_1=\sum_{\sigma_{d-2}}\dot\theta_{\sigma_{d-2}}^{\sigma_d}V_{\sigma_{d-2}} o_{\sigma_{d-2}}+\sum_{\sigma_{d-2}\subset \sigma_{d-1}}V_{\sigma_{d-2}} n_{\sigma_{d-2}}^{\sigma_{d-1}}\dot h_{\sigma_{d-1}}^{\sigma_{d}}.
\end{equation*}
The second sum is zero by applying the Minkowski balancing property, \cref{res:Minkowski_balancing}, to every $\sigma_{d-1}$. Hence,
\begin{equation}
\label{T1}
T_1=\sum_{\sigma_{d-2}}\dot\theta_{\sigma_{d-2}}^{\sigma_d}V_{\sigma_{d-2}} o_{\sigma_{d-2}}.
\end{equation}

Now it remains to handle $T_2$. For every $k$ with $0\leq k \leq d-2$ and every full flag $f$, there exists a unique other full flag $f^*$ coinciding with $f$ in all dimensions but $k+1$. By applying the Minkowski balancing property to the quadrilateral $o_{f_k}o_{f_{k+1}}o_{f_{k+1}^\ast}o_{f_{k+2}}$, we obtain
\[h_{f_{k}}n_{f_{k+1}}+h_{f_{k}^\ast}n_{f_{k+1}^\ast}=n_{f_{k}}h_{f_{k+1}}+n_{f_{k}^\ast}h_{f_{k+1}^\ast}.\]
From such pairings of flags, we deduce
\begin{equation}
\label{afterswing}
T_2= \dfrac{1}{(d-1)!}  \sum_f \sum_{k=0}^{d-3}(k+1) h_{f_0}  h_{f_1} \ldots h_{f_{k-1}} \left(h_{f_{k}} ^{f_{k+2}}\right)^2   \dot{\alpha}_{f_{k+2}} {n}_{f_{k+2}}  h_{f_{k+3}}    \ldots h_{f_{d-1}}.
\end{equation}

The Schläfli formula, \cref{res:schlafli}, says that for every $\sigma_{k+2}$ we have
\begin{equation*}
\label{schla}
\sum_{\sigma_k: \sigma_k \subset \sigma_{k+2}} V_{\sigma_k}\dot \theta_{\sigma_k}^{\sigma_{k+2}}=0.
\end{equation*}
Together with~\eqref{alphatheta} and~\eqref{volume}, it implies that for every $\sigma_{k+2}$ we obtain
\[
    \sum_{f^{k+2}: f^{k+2}_{k+2}=\sigma_{k+2}}
h_{f^{k+2}_0} \ldots h_{f^{k+2}_{k-1}}\dot\alpha_{f^{k+2}_{k+2}}=0.
\]

From the Pythagorean theorem for the triangle $o_{ f_{k}} o_{f_{k+2}}o_{f_{k+3}}$, we get 
\[
    \left(h_{f_{k}} ^{f_{k+2}}\right)^2=\left(h_{f_{k}} ^{f_{k+3}}\right)^2-\left(h_{f_{k+2}} ^{f_{k+3}}\right)^2.
\]
Using the last two equations in~\eqref{afterswing}, we have
\begin{equation}\label{eq:term2}
T_2=\dfrac{1}{(d-1)!} \sum_{f}\sum_{k=0}^{d-3}(k+1) h_{f_0} \ldots h_{f_{k-1}}\left(h_{f_{k}} ^{f_{k+3}}\right)^2 \dot{\alpha}_{f_{k+2}}    n_{f_{k+2}}h_{f_{k+3}}   \ldots h_{f_{d-1}}.
\end{equation}

Fix $k$ with $0\leq k \leq d-3$ and fix a 
%
%
%
partial flag $F$ that has a face in every dimension but $k+1$ and $k+2$. Notice that \eqref{glucklink} and the pairings of flags imply that
\begin{equation}
\label{gluckflag}
\sum_{f: F \subset f}  \dot{\alpha} _{f_{k+2}}{n}_{f_{k+2}}  =\sum_{f: F \subset f} {n}_{f_{k}}\dot{\alpha}_{f_{k+3}}  .
\end{equation} 
We apply~\eqref{gluckflag} in~\eqref{eq:term2} and obtain
\begin{equation}\label{term3}
T_2=\dfrac{1}{(d-1)!} \sum_{f}\sum_{k=0}^{d-3}(k+1) h_{f_0} \ldots h_{f_{k-1}}\left(h_{f_{k}} ^{f_{k+3}}\right)^2 n_{f_{k}}\dot{\alpha}_{f_{k+3}}    h_{f_{k+3}}   \ldots h_{f_{d-1}}.
\end{equation}

The Minkowski balancing property says that for every $\sigma_{k+1}$ we have
\begin{equation*}
\sum_{f^{k+1}: f^{k+1}_{k+1}=\sigma_{k+1}}
h_{f^{k+1}_0} \ldots h_{f^{k+1}_{k-1}}n_{f_k}=0.
\end{equation*}

From the Pythagorean theorem for the triangle $o_{ f_{k}} o_{f_{k+1}}o_{f_{k+3}}$, we get 
\[
    \left(h_{f_{k}} ^{f_{k+3}}\right)^2=\left(h_{f_{k}} \right)^2+\left(h_{f_{k+1}}^{f_{k+3}}\right)^2.
\]
Using the last two equations in~\eqref{term3}, we get
\begin{equation}\label{term4}
T_2=\dfrac{1}{(d-1)!} \sum_{f}\sum_{k=0}^{d-3}(k+1) h_{f_0} \ldots h_{f_{k-1}}\left(h_{f_{k}}\right)^2     n_{f_{k}}\dot{\alpha}_{f_{k+3}}h_{f_{k+3}}   \ldots h_{f_{d-1}}.
\end{equation}

For a full flag $f$ and $k$ with $0\leq k\leq d-1$, we have $\langle {n}_{f_{k}}, {o}_{f_{k+1}}\rangle=0$. For any $l\leq k$, $o_{f_l}$ belongs to the affine span of $f_k$. Hence, it belongs to the hyperplane orthogonal to $n_{f_k}$ at the oriented distance $h_{f_k}$ from $o_{f_{k+1}}$. This implies that for any $l\leq k$,
\begin{equation}
\label{scalar}
h_{f_{k}}=\langle {n}_{f_{k}}, {o}_{f_{l}}-{o}_{f_{k+1}}\rangle=\langle n_{f_{k}}, o_{f_{l}} \rangle.
\end{equation}

Now fix $k$ with $0\leq k \leq d-4$ and fix a partial flag $F$ that has a face in every dimension but $k+2$ and $k+3$.
Combining~\eqref{scalar} and~\eqref{gluckflag}, we obtain 
\begin{multline*}
\sum_{f: F \subset f}\dot\alpha_{f_{k+3}} h_{f_{k+3}}=\sum_{f: F \subset f}\langle \dot\alpha_{f_{k+3}} n_{f_{k+3}}, o_{f_{k+1}} \rangle =\\
=\sum_{f: F \subset f} \langle {n}_{f_{k+1}}\dot{\alpha}_{f_{k+4}}, o_{f_{k+1}} \rangle = \sum_{f: F \subset f}h_{f_{k+1}}\dot\alpha_{f_{k+4}}.     
\end{multline*}
Note that here we use that $f_{k+1}$ is the same for all $f$ containing $F$.

We apply it in~\eqref{term4} and get
\begin{equation}\label{T2}
T_2=\dfrac{1}{(d-1)!} \sum_{f}\sum_{k=0}^{d-3}(k+1) h_{f_0} \ldots h_{f_{k-1}}\left(h_{f_{k}}\right)^2     n_{f_{k}}h_{f_{k+1}}   \ldots h_{f_{d-3}}\dot{\alpha}_{f_{d}}.
\end{equation}

Because of \eqref{alphatheta}, we see that if for all $\sigma_{d-2}$ our first-order deformation satisfies $\dot\theta_{\sigma_{d-2}}^{\sigma_d}=0$, then $T_2=0$. Additionally, from \eqref{T1}, in such case also $T_1=0$. This finishes the proof of \eqref{Stoker_type_eq}.

\subsection{Proof of \cref{res:SS*}}

Now we deal with the general case. To this purpose, we unpack the right-hand side of~\eqref{mains}. 

For us, the \Def{barycenter} of a finite point set $p_1, \ldots, p_n \in \RR^d$ is the vector $\frac{1}{n}\sum_{i=0}^n p_i$. The barycenter of a polytope is the barycenter of its vertex set.

For $l$ with $0\leq l \leq d$ and an $l$-order flag $f^l$, denote by $b_{f^l}$ the barycenter of the simplex spanned by $o_{f^l_0}, \ldots, o_{f^l_l}$. Note that $b_{f^l}:=\frac{1}{l+1}\sum_{k=0}^l o_{f_k^l}$. Next, for $k$ with $0\leq k \leq l-1$, \[o_{f_k^l}-o_{f_l^l}=\sum_{m=k}^{l-1}(o_{f_m^l}-o_{f_{m+1}^l})=\sum_{m=k}^{l-1} h_{f_m^l}n_{f_m^l}.\] Hence,
\begin{equation*}
b_{f^l}-o_{f_l^l}=\frac{1}{l+1}\sum_{k=0}^{l-1} (k+1)h_{f_k^l}n_{f_k^l}.
\end{equation*}

Let $P_1$ and $P_2$ be two full-dimensional polytopes in $\RR^d$ with disjoint interiors, $V_1$ and $V_2$ be their volumes, $b_1$ and $b_2$ be their barycenters and $b$ be the 
barycenter of their union. By the basic property of barycenters,
\[(V_1+V_2)b=V_1b_1+V_2b_2.\]

In turn, this means that for every face $\sigma_l$, we have
\begin{multline*}
V_{\sigma_l}(b_{\sigma_l}-o_{\sigma_l})=\sum_{f^l: f_l^l=\sigma_l} V_{f^l} (b_{f^l}-o_{\sigma_l})=\\
=\frac{1}{l!}\sum_{f^l: f_l^l=\sigma_l} h_{f_0^l}\ldots h_{f_{l-1}^l} \frac{1}{l+1}\sum_{k=0}^{l-1} (k+1) h_{f_k^l} n_{f_k^l}=\\
=\frac{1}{(l+1)!}\sum_{f^l: f_l^l=\sigma_l}\sum_{k=0}^{l-1}(k+1) h_{f_0^l}\ldots h_{f_{k-1}^l}(h_{f_k^l})^2n_{f_k^l} h_{f_{k+1}^l}\ldots h_{f_{l-1}^l}.
\end{multline*}

By applying this to every $\sigma_{d-2}$, from~\eqref{T2}, we conclude that
\[T_2=\sum_{\sigma_{d-2}}\dot\theta_{\sigma_{d-2}}^{\sigma_d} V_{\sigma_{d-2}}(b_{\sigma_{d-2}}-o_{\sigma_{d-2}}).\]
Together with~\eqref{T1}, this shows that
\[T_1+T_2=\sum_{\sigma_{d-2}}\dot\theta_{\sigma_{d-2}}^{\sigma_d} V_{\sigma_{d-2}}b_{\sigma_{d-2}},\]
which finishes the proof.

\tempnewpage

\section{Conclusion and outlook}
\label{sec:outlook}

The three main achievements of this article have been
\begin{itemize}
    \item a proof that coned polytope frameworks are second-order rigid (and in fact, prestress stable) (\cref{res:prestress_stable}, proven in \cref{sec:stress_flex_to_CPF}),
    \item a resolution of the weak stress-flex conjecture (\cref{res:weak_stress_flex}, proven in \cref{sec:SS}), and
    \item a new, self-contained and completely discrete-geometric  proof of the vector-valued Schläfli formula by Schlenker and Souam (\cref{res:SS*}, proven in \cref{sec:Stoker_type}).
\end{itemize}
We found it remarkable that these three topics ended up so tightly related: prestress stability of a CPF is a direct consequence of the weak stress flex conjecture, which, in turn, is dual to a special case of the vector-valued Schläfli formula.

As mentioned before, the version of the stress-flex conjecture resolved here was merely the fragment required to prove the second-order rigidity result.
Experiments suggest that a version of it holds in far greater generality:

\begin{conjecture}[Strong stress-flex conjecture]
    \label{conj:stress_flex}
    \label{conj:strong_stress_flex}
    Let $(G_S^\star, \bs p)$ be the coned framework of a closed piecewise linear (PL) surface $S$ (\ie\ the vertices and edges of the framework are the vertices and edges of the surface $S$; and there is an additional cone point $p_*$ together with cone edges connecting it to the vertices of $S$).
    If $\dotbs p$ is any first-order flex with $\dot p_\star=0$, and $\bs\omega$ is any stress of $(G_S^\star, \bs p)$, then
    $$\sum_i \omega_{\star i} \dot p_i=0.$$
\end{conjecture}

This formulation generalizes the weak stress-flex conjecture in three ways.
First, the convex polytope is replaced by a general PL surface. This surface might be~non-convex, of higher genus, self-intersecting or even non-orientable.
Second, the cone point might lie outside the region enclosed by the surface (if there is a well-defined such region at all).
Third, the Wachspress stress is replaced by an arbitrary stress of the coned framework.
Below we give some comments on these generalizations.

\subsection{General surfaces}
\label{sec:non_convex}

A version of the Wachspress stress exists for any closed oriented PL surface (making only mild non-degeneracy assumptions).
The Wachs\-press stress as defined in \cref{sec:WI_stress} is based on polar duality.
Hence, a suitable notion of polarity for general PL surfaces is required.
Likewise, the proof of \cref{res:Stoker_type} can be extended to cover general PL surfaces.
In fact, Schlenker and Souam already comment on this in \cite[Section 6]{SS2}, though not quite in the generality we comment on here.
Once provided, these ideas already suffice to prove \cref{conj:stress_flex} for~general closed oriented PL surfaces at least in the case of the Wachspress stress.

The details of what we describe above are technical, and will be part of a future publication.

\subsection{General stresses}
\label{sec:general_stresses}

Our proof of the weak stress-flex conjecture relies heavily on the precise form of the Wachspress stress and does not, in any obvious way,\nls generalize to general stresses of coned frameworks.
This is already true if the surface~$S$ is a convex polytope.

The Wachspress stress is the only stress of a \textit{simple} PL surface, but other stresses emerge once we add more edges. 
While adding edges, the number of first-order~flexes decreases. 
There exists a sweet-spot where we have both first-order flexes but~also multiple stresses.
These constitute the difficult instances of \cref{conj:stress_flex}.

\begin{figure}
    \centering
    \includegraphics[width=0.23\textwidth]{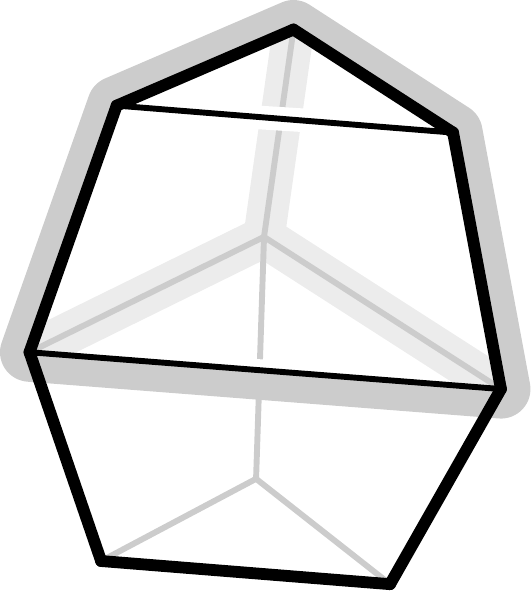}
    \caption{A convex polytope $P$ and a highlighted subset $P'$ that is a polytope in itself. The edges of $P'$ are a subset of the edges of $P$.}
    \label{fig:subsurface}
\end{figure}

One idea is trying to understand a general stress of a CPF (or a general coned PL surface) as a combination of
\begin{myenumerate}
    \item the Wachspress stress,
    \item Wachspress stresses of suitable sub-surfaces, and
    \item generic stresses.
\end{myenumerate}

An example of \itm2 is given in \cref{fig:subsurface}.
The highlighted part $P'$ of the depicted polytope $P$ forms a polytope in itself. Crucially, all edges of $P'$ are already edges of $P$, and hence, a cone over the 1-skeleton of $P$ carries both the Wachspress stress of $P$ and of $P'$ as separate stresses. 
Both satisfy stress-flex orthogonality.

Certain first-order flexes and stresses of $(G,\bs p)$ exist generically, in the sense that they are limits of first-order flexes or stresses in a sequence of generic frameworks that converges to $(G,\bs p)$.
The following argument was provided by Sean Dewar~and establishes stress-flex orthogonality between such generic stresses (in the sense of \itm3) and generic first-order flexes.

\begin{remark}[Generic stress-flex orthogonality]

Consider the \emph{measurement map}
$$f_G:\RR^{dV}\to\RR^E,\; \bs p\mapsto (\|p_j-p_i\|^2)_{ij\in E}.$$
If $\bs p$ is generic, then $f_G(\bs p)$ is a regular point of $\im(f_G)$.
This means that $\im(f_G)$ is a smooth manifold in a neighborhood $U\subset\RR^E$ of $f_G(\bs p)$.
In particular, the tangent space is well defined at $f_G(\bs p)$.
An elementary exercise shows that $\bs\omega\in\RR^E$ is a stress of $(G,\bs p)$ if and only if it is orthogonal to this tangent space of $\im(f_G)$, or equivalently, if $\bs\omega\T Df_G(\bs p)=0$.
If $(G,\bs q)$ is any framework equivalent to $(G,\bs p)$, then they define the same point $f_G(\bs p)=f_G(\bs q)$ in $\im(f_G)$, and hence have the same tangent space and therefore also the exact same stresses.

For a stress $\bs\omega$, consider
$$L_{\bs\omega}:=\{\bs q\in\RR^{dV}\mid \text{$\bs\omega$ is a stress of $\bs q$}\}.$$
We just argued that $\real(G,\bs p):=f_G^{-1}(f_G(\bs p))\subseteq L_{\bs \omega}$.
Also here, since $\bs p$ is generic, $\real(G,\bs p)$ is smooth in a neighborhood $W$ of $\bs p$ with a well-defined tangent space.
Then also the tangent space lies in $L_{\bs\omega}$.
Once again, an elementary exercise shows that the elements of this tangent space are precisely the first-order flexes of $(G,\bs p)$.
Hence, if $\dotbs p$ is such a first-order flex, it is also in $L_{\bs\omega}$.
Reinterpreting $(G,\dotbs p)$ as a framework with stress $\bs\omega$, for each $i\in V$ we obtain
$$\sum_{j_j\sim i}\omega_{ij}(\dot p_j-\dot p_i)=0.$$
If we choose $i=\star$ and assume $\dot p_\star=0$, this becomes stress-flex orthogonality.

\end{remark}

\subsection{Other Schläfli-type formulas}

Another interesting question regarding the general content of the this article relates to our proof of \cref{res:SS} and to other Schläfli-type formulas. First, as we mentioned, in~\cite{SS2} Schlenker and Souam also prove the spherical and hyperbolic counterparts to \cref{res:SS}. It is of interest to understand, whether our approach can generalize to these settings using the quadric models of spherical and hyperbolic $d$-spaces.

Second, in~\cite{SS2} Schlenker and Souam additionally show higher-signature vector-valued generalizations of \cref{res:SS*} in all constant-curvature ambient geometries, which extend their earlier higher-signature scalar analogues of the Schläfli formula obtained in~\cite{SS}. It is of interest, whether our methods can be also applied to give new proofs of these results.

\bibliographystyle{abbrv}
\bibliography{literature}

@article{izmestiev2010colin,
  title={The {C}olin de {V}erdiere number and graphs of polytopes},
  author={Izmestiev, Ivan},
  journal={Israel Journal of Mathematics},
  volume={178},
  number={1},
  pages={427--444},
  year={2010},
  publisher={Springer}
}

@book{schneider2013convex,
  title={Convex bodies: the Brunn--Minkowski theory},
  author={Schneider, Rolf},
  volume={151},
  year={2013},
  publisher={Cambridge university press}
}

@article{winter2024rigidity,
  title={Rigidity, tensegrity, and reconstruction of polytopes under metric constraints},
  author={Winter, Martin},
  journal={International Mathematics Research Notices},
  volume={2024},
  number={9},
  pages={7721--7747},
  year={2024},
  publisher={Oxford University Press}
}

@misc{winter2024stressflex,
    title={The Stress-Flex Conjecture}, 
    author={Robert Connelly and Steven J. Gortler and Louis Theran and Martin Winter},
    year={2024},
    eprint={2404.15590},
    archivePrefix={arXiv},
    primaryClass={math.CO},
    url={https://arxiv.org/abs/2404.15590}
}

@article{connelly1996second,
  title={Second-order rigidity and prestress stability for tensegrity frameworks},
  author={Connelly, Robert and Whiteley, Walter},
  journal={SIAM Journal on Discrete Mathematics},
  volume={9},
  number={3},
  pages={453--491},
  year={1996},
  publisher={SIAM}
}

@incollection {AVS,
    AUTHOR = {Alekseevskij, D. V. and Vinberg, \`E.\ B. and Solodovnikov, A. S.},
     TITLE = {Geometry of spaces of constant curvature},
 BOOKTITLE = {Geometry, {II}},
     PAGES = {1--138},
 PUBLISHER = {Springer, Berlin},
      YEAR = {1993},
      ISBN = {3-540-52000-7},
   MRCLASS = {53C20 (22E40 57M50)},
  MRNUMBER = {1254932},
MRREVIEWER = {Alan\ W.\ Reid},
       DOI = {10.1007/978-3-662-02901-5\_1},
       URL = {https://doi.org/10.1007/978-3-662-02901-5_1},
}

@incollection {Mil,
    AUTHOR = {Milnor, John},
     TITLE = {The {S}chl\"afli differential equality},
 BOOKTITLE = {Collected papers. {V}ol. 1},
     PAGES = {281--295},
 PUBLISHER = {Publish or Perish, Inc., Houston, TX},
      YEAR = {1994},
}

@article {BI,
    AUTHOR = {Bobenko, Alexander I. and Izmestiev, Ivan},
     TITLE = {Alexandrov's theorem, weighted {D}elaunay triangulations, and
              mixed volumes},
   JOURNAL = {Ann. Inst. Fourier (Grenoble)},
  FJOURNAL = {Universit\'e{} de Grenoble. Annales de l'Institut Fourier},
    VOLUME = {58},
      YEAR = {2008},
    NUMBER = {2},
     PAGES = {447--505},
      ISSN = {0373-0956,1777-5310},
   MRCLASS = {52B10 (52A39 52C25)},
  MRNUMBER = {2410380},
MRREVIEWER = {Fran\c cois\ Fillastre},
       DOI = {10.5802/aif.2358},
       URL = {https://doi.org/10.5802/aif.2358},
}

@article {Pro,
    AUTHOR = {Prosanov, Roman},
     TITLE = {Rigidity of compact {F}uchsian manifolds with convex boundary},
   JOURNAL = {Int. Math. Res. Not. IMRN},
  FJOURNAL = {International Mathematics Research Notices. IMRN},
      YEAR = {2023},
    NUMBER = {3},
     PAGES = {1959--2094},
      ISSN = {1073-7928,1687-0247},
   MRCLASS = {57K32},
  MRNUMBER = {4565606},
MRREVIEWER = {Yasushi\ Yamashita},
       DOI = {10.1093/imrn/rnab270},
       URL = {https://doi.org/10.1093/imrn/rnab270},
}

@article {BPS,
    AUTHOR = {Bobenko, Alexander I. and Pinkall, Ulrich and Springborn,
              Boris A.},
     TITLE = {Discrete conformal maps and ideal hyperbolic polyhedra},
   JOURNAL = {Geom. Topol.},
  FJOURNAL = {Geometry \& Topology},
    VOLUME = {19},
      YEAR = {2015},
    NUMBER = {4},
     PAGES = {2155--2215},
      ISSN = {1465-3060,1364-0380},
   MRCLASS = {52C26 (57M50)},
  MRNUMBER = {3375525},
MRREVIEWER = {Gunter\ Semmler},
       DOI = {10.2140/gt.2015.19.2155},
       URL = {https://doi.org/10.2140/gt.2015.19.2155},
}

@article {IPW,
    AUTHOR = {Izmestiev, Ivan and Prosanov, Roman and Wu, Tianqi},
     TITLE = {Prescribed curvature problem for discrete conformality on
              convex spherical cone-metrics},
   JOURNAL = {Adv. Math.},
  FJOURNAL = {Advances in Mathematics},
    VOLUME = {437},
      YEAR = {2024},
     PAGES = {Paper No. 109439, 35},
      ISSN = {0001-8708,1090-2082},
   MRCLASS = {30F45 (52B10 52B70 52C26)},
  MRNUMBER = {4674859},
MRREVIEWER = {Byung-Geun\ Oh},
       DOI = {10.1016/j.aim.2023.109439},
       URL = {https://doi.org/10.1016/j.aim.2023.109439},
}

@article{Riv,
    AUTHOR = {Rivin, Igor},
     TITLE = {Euclidean structures on simplicial surfaces and hyperbolic
              volume},
   JOURNAL = {Ann. of Math. (2)},
  FJOURNAL = {Annals of Mathematics. Second Series},
    VOLUME = {139},
      YEAR = {1994},
    NUMBER = {3},
     PAGES = {553--580},
      ISSN = {0003-486X,1939-8980},
   MRCLASS = {57M50 (57Q99)},
  MRNUMBER = {1283870},
       DOI = {10.2307/2118572},
       URL = {https://doi.org/10.2307/2118572},
}

@article {BLP,
    AUTHOR = {Boileau, Michel and Leeb, Bernhard and Porti, Joan},
     TITLE = {Geometrization of 3-dimensional orbifolds},
   JOURNAL = {Ann. of Math. (2)},
  FJOURNAL = {Annals of Mathematics. Second Series},
    VOLUME = {162},
      YEAR = {2005},
    NUMBER = {1},
     PAGES = {195--290},
      ISSN = {0003-486X,1939-8980},
   MRCLASS = {57M50 (53C23 57N10)},
  MRNUMBER = {2178962},
MRREVIEWER = {Darryl\ McCullough},
       DOI = {10.4007/annals.2005.162.195},
       URL = {https://doi.org/10.4007/annals.2005.162.195},
}

@article {Luo,
    AUTHOR = {Luo, Feng},
     TITLE = {Volume and angle structures on 3-manifolds},
   JOURNAL = {Asian J. Math.},
  FJOURNAL = {Asian Journal of Mathematics},
    VOLUME = {11},
      YEAR = {2007},
    NUMBER = {4},
     PAGES = {555--566},
      ISSN = {1093-6106,1945-0036},
   MRCLASS = {57M50 (57M30 57N10)},
  MRNUMBER = {2402938},
MRREVIEWER = {Bruno\ P.\ Zimmermann},
       DOI = {10.4310/AJM.2007.v11.n4.a2},
       URL = {https://doi.org/10.4310/AJM.2007.v11.n4.a2},
}

@article {SS,
    AUTHOR = {Schlenker, Jean-Marc and Souam, Rabah},
     TITLE = {Higher {S}chl\"afli formulas and applications},
   JOURNAL = {Compositio Math.},
  FJOURNAL = {Compositio Mathematica},
    VOLUME = {135},
      YEAR = {2003},
    NUMBER = {1},
     PAGES = {1--24},
      ISSN = {0010-437X,1570-5846},
   MRCLASS = {53A07 (52B11)},
  MRNUMBER = {1955161},
MRREVIEWER = {Igor\ Rivin},
       DOI = {10.1023/A:1021787011407},
       URL = {https://doi.org/10.1023/A:1021787011407},
}

@article {SS2,
    AUTHOR = {Schlenker, Jean-Marc and Souam, Rabah},
     TITLE = {Higher {S}chl\"afli formulas and applications. {II}.
              {V}ector-valued differential relations},
   JOURNAL = {Int. Math. Res. Not. IMRN},
  FJOURNAL = {International Mathematics Research Notices. IMRN},
      YEAR = {2008},
     PAGES = {Art. ID rnn 068, 44},
      ISSN = {1073-7928,1687-0247},
   MRCLASS = {53A07 (52B11 52B70)},
  MRNUMBER = {2439567},
MRREVIEWER = {Jesse\ Ratzkin},
       DOI = {10.1093/imrn/rnn068},
       URL = {https://doi.org/10.1093/imrn/rnn068},
}

@book {Ale,
    AUTHOR = {Alexandrov, A. D.},
     TITLE = {Convex polyhedra},
 PUBLISHER = {Springer-Verlag, Berlin},
      YEAR = {2005},
     PAGES = {xii+539},
      ISBN = {3-540-23158-7},
   MRCLASS = {52-02 (52Axx 52Bxx)},
  MRNUMBER = {2127379},
}

@incollection {Glu,
    AUTHOR = {Gluck, Herman},
     TITLE = {Almost all simply connected closed surfaces are rigid},
 BOOKTITLE = {Geometric topology ({P}roc. {C}onf., {P}ark {C}ity, {U}tah,
              1974)},
     PAGES = {225--239},
 PUBLISHER = {Springer, Berlin-New York},
      YEAR = {1975},
   MRCLASS = {57C05 (53C45 57A05)},
  MRNUMBER = {400239},
MRREVIEWER = {Robert\ Connelly},
}

@article {Sch,
    AUTHOR = {Schlenker, Jean-Marc},
     TITLE = {Small deformations of polygons and polyhedra},
   JOURNAL = {Trans. Amer. Math. Soc.},
  FJOURNAL = {Transactions of the American Mathematical Society},
    VOLUME = {359},
      YEAR = {2007},
    NUMBER = {5},
     PAGES = {2155--2189},
      ISSN = {0002-9947,1088-6850},
   MRCLASS = {58D29 (51M16 52B99)},
  MRNUMBER = {2276616},
MRREVIEWER = {Igor\ Rivin},
       DOI = {10.1090/S0002-9947-06-04172-9},
       URL = {https://doi.org/10.1090/S0002-9947-06-04172-9},
}

@article {GBKKRS,
    AUTHOR = {Grosse-Brauckmann, Karsten and Korevaar, Nicholas J. and
              Kusner, Robert B. and Ratzkin, Jesse and Sullivan, John M.},
     TITLE = {Coplanar {$k$}-unduloids are nondegenerate},
   JOURNAL = {Int. Math. Res. Not. IMRN},
  FJOURNAL = {International Mathematics Research Notices. IMRN},
      YEAR = {2009},
    NUMBER = {18},
     PAGES = {3391--3416},
      ISSN = {1073-7928,1687-0247},
   MRCLASS = {53A10},
  MRNUMBER = {2535004},
MRREVIEWER = {Rafael\ L\'opez},
       DOI = {10.1093/imrn/rnp058},
       URL = {https://doi.org/10.1093/imrn/rnp058},
}

@article{asimow1978rigidity,
  title={The rigidity of graphs},
  author={Asimow, Leonard and Roth, Ben},
  journal={Transactions of the American Mathematical Society},
  volume={245},
  pages={279--289},
  year={1978}
}

@article{connelly1980rigidity,
  title={The rigidity of certain cabled frameworks and the second-order rigidity of arbitrarily triangulated convex surfaces},
  author={Connelly, Robert},
  journal={Advances in Mathematics},
  volume={37},
  number={3},
  pages={272--299},
  year={1980},
  publisher={Academic Press}
}

@article{connelly2017prestress,
  title={Prestress stability of triangulated convex polytopes and universal second-order rigidity},
  author={Connelly, Robert and Gortler, Steven J},
  journal={SIAM Journal on Discrete Mathematics},
  volume={31},
  number={4},
  pages={2735--2753},
  year={2017},
  publisher={SIAM}
}

@book{wachspress1975rational,
 author = {Wachspress, Eugene L.},
 title = {A rational finite element basis},
 year = {1975},
 publisher = {Elsevier, Amsterdam},
 language = {English},
 zbMATH = {3504364},
 Zbl = {0322.65001}
}

@article{warren1996barycentric,
  title={Barycentric coordinates for convex polytopes},
  author={Warren, Joe},
  journal={Advances in Computational Mathematics},
  volume={6},
  number={1},
  pages={97--108},
  year={1996},
  publisher={Springer}
}

@article{kohn2020projective,
  title={Projective geometry of {W}achspress coordinates},
  author={Kohn, Kathl{\'e}n and Ranestad, Kristian},
  journal={Foundations of Computational Mathematics},
  volume={20},
  pages={1135--1173},
  year={2020},
  publisher={Springer}
}

@article{irving2014geometry,
  title={Geometry of {W}achspress surfaces},
  author={Irving, Corey and Schenck, Hal},
  journal={Algebra \& Number Theory},
  volume={8},
  number={2},
  pages={369--396},
  year={2014},
  publisher={Mathematical Sciences Publishers}
}

@article{lovasz2001steinitz,
  title={{S}teinitz representations of polyhedra and the {C}olin de {V}erdiere number},
  author={Lov{\'a}sz, L{\'a}szl{\'o}},
  journal={Journal of Combinatorial Theory, Series B},
  volume={82},
  number={2},
  pages={223--236},
  year={2001},
  publisher={Elsevier}
}

@article{narayanan2021spectral,
 author = {Narayanan, Hariharan and Shah, Rikhav and Srivastava, Nikhil},
 title = {A spectral approach to polytope diameter},
 fjournal = {Discrete \& Computational Geometry},
 journal = {Discrete Comput. Geom.},
 issn = {0179-5376},
 volume = {72},
 number = {4},
 pages = {1647--1674},
 year = {2024},
 language = {English},
 doi = {10.1007/s00454-024-00636-y},
 zbMATH = {7948165},
 Zbl = {1552.60028}
}

@misc{winter2026secondorder,
      title={Deformations and second-order rigidity of polytopes}, 
      author={Matthias Adrian-Himmelmann and Martin Winter and Zhen Zhang},
      year={2026},
      eprint={2607.09252},
      archivePrefix={arXiv},
      primaryClass={math.CO},
      url={https://arxiv.org/abs/2607.09252}, 
}
\addresseshere

\newpage

\appendix

\section{Proof of \cref{res:gluck}}
\label{sec:gluck}

For convenience, we repeat the setup of \cref{res:gluck}:

%
\begin{theoremX}{\ref{res:gluck}}
Let $n_1, \ldots, n_k$ be an ordered $k$-tuple of points on $\S^2 \subset \RR^3$ with $n_i \neq \pm n_{i-1}$, $\alpha_i$ be the oriented angle $n_{i-1}n_in_{i+1}$ (the angle between spherical segments) with $\alpha_i \not\in\{ 0,\pm\pi\}$, $i\in\{1, ..., k\}$. Let $m_i$ be the dual point~to~the oriented segment $n_in_{i+1}$ and $\beta_i$ be the oriented angle $m_{i-1}m_i m_{i+1}$ (hence, $\beta_i$ is equal to the oriented length of $n_i n_{i+1}$). Then for a
first-order deformation of this configuration we get
\[\sum_i \dot\alpha_i n_i =\sum_i\dot\beta_i m_i.\]
\end{theoremX}

\begin{proof}
Here we call a first-order deformation \Def{trivial} if it arises from the restriction of a Killing field of $\S^2$ to $\{n_i\}$. We call two first-order deformations \Def{equivalent} if they differ by a trivial deformation. Given a first-order deformation, clearly it suffices to prove \cref{res:gluck} for any equivalent deformation.

An easy observation is that a first-order deformation of our configuration is determined up to equivalence by the first-order changes of the lengths of the segments $n_1 n_2, \ldots, n_{k-1} n_{k}$ (so, all but one) and by the first-order changes of the angles at $n_2, \ldots, n_{k-1}$ (so, all but two). 

Let 
\[\Phi: \RR^3 \rightarrow \mathfrak{so}(3)\]
be the cross-product isomorphism, where we perceive $\mathfrak{so}(3)$ as the space of the Killing fields on $\S^2$. Namely, we send $x \in \RR^3$ to the vector field $\Phi(x)$ on $\S^2 \subset \RR^3$ defined by $\Phi(x)(y):=y \times x$ for $y \in \S^2$, where $\times$ is the Euclidean cross product in $\RR^3$. Define $a_1$ to be the zero Killing field and define
\[a_i:=\sum_{j=1}^{i-1}\Phi(\dot\alpha_j n_j -\dot\beta_j m_j).\]
We claim that the first-order deformation $\{a_i(n_i)\}$ is equivalent to the initial first-order deformation. Indeed, 
\[a_{i+1}-a_i=\Phi(\dot\alpha_i n_i -\dot\beta_i m_i),~~~~~i=1,\ldots,n-1.\]
By evaluating $a_{i+1}-a_i$ at $n_{i+1}$, we see that this choice of vectors indeed induces the first-order change $\dot\beta_i$ on the length of $n_i n_{i+1}$, $i=1,\ldots,k-1$. By evaluating $a_{i+1}-a_i$ at $n_{i+1}$ and $a_{i-1}-a_i$ at $n_{i-1}$, we see that this choice of vectors indeed induces the first-order change $\dot\alpha_i$ on the angle $n_{i-1}n_i n_{i+1}$, $i=2, \ldots, k-1$. Hence, the chosen first-order deformation is indeed equivalent to the initial one. Now, by evaluating $a_1-a_k$ and evaluating $\Phi(\dot \alpha_k n_k-\dot\beta_k m_k)$ at $ n_1$ and $ n_k$, we see that they have the same values at $ n_1$ and $ n_k$, hence
\[a_1-a_k=\Phi(\dot \alpha_k n_k-\dot\beta_k m_k).\]
Substituting there $a_1=0$ and 
$a_k=\sum_{j=1}^{k-1}\Phi(\dot\alpha_j n_j -\dot\beta_j m_j),$
we get
\[\Phi\left(\sum_{j=1}^{k}\dot\alpha_j n_j -\dot\beta_j m_j\right)=0,\]
from which \cref{res:gluck} follows.
\end{proof}

\end{document}